\documentclass[leqno]{amsart}
\usepackage{amsmath,amsthm,amssymb,latexsym,graphicx, mathtools, enumerate,bm,cite,color}

\newcommand{\NN}{{\mathbb N}}

\newcommand{\QQ}{{\mathbb Q}}
\newcommand{\RR}{{\mathbb R}}

\newcommand{\abs}[1]{ \left| #1 \right|}
\newcommand{\bs}[1]{\boldsymbol{#1}}
\newcommand{\del}{\partial}

\newcommand{\eps}{\varepsilon}

\newcommand{\ind}{\mathbf{1}}

\newcommand{\nor}[2]{\left\|#1\right\|_{#2}}
\newcommand{\oline}[1]{\overline{#1}}
\newcommand{\oo}{\infty}
\newcommand{\pars}[1]{\left(#1\right)}

\newcommand{\tr}{\on{tr}}

\DeclareMathOperator*{\argmax}{arg\,max}

\newcommand{\mcl}{\mathcal}

\newcommand{\mbf}{\mathbf}
\newcommand{\on}{\operatorname}

\newtheorem{lemma}{Lemma}
\newtheorem{proposition}{Proposition}
\newtheorem{theorem}{Theorem}

\theoremstyle{definition}

\numberwithin{equation}{section}
\numberwithin{lemma}{section}
\numberwithin{proposition}{section}
\numberwithin{theorem}{section}
\numberwithin{corollary}{section}
\numberwithin{definition}{section}

\begin{document}
\title{Homogenization of a stochastically forced Hamilton-Jacobi equation}
\author{Benjamin Seeger}

\address{Universit\'e Paris-Dauphine \& Coll\`ege de France \\ Place du Mar\'echal de Lattre de Tassigny \\ 75016 Paris, France}
\email{seeger@ceremade.dauphine.fr}

\thanks{Partially supported by the National Science Foundation Mathematical Sciences Postdoctoral Research Fellowship under Grant Number DMS-1902658}


\date{\today}

\maketitle

\begin{abstract}
	We study the homogenization of a Hamilton-Jacobi equation forced by rapidly oscillating noise that is colored in space and white in time. It is shown that the homogenized equation is deterministic, and, in general, the noise has an enhancement effect, for which we provide a quantitative estimate. As an application, we perform a noise sensitivity analysis for Hamilton-Jacobi equations forced by a noise term with small amplitude, and identify the scaling at which the macroscopic enhancement effect is felt. The results depend on new, probabilistic estimates for the large scale H\"older regularity of the solutions, which are of independent interest. 
\end{abstract}

\section{Introduction}

The purpose of this paper is to study the asymptotic behavior of stochastically forced Hamilton-Jacobi equations that take the form
\begin{equation} \label{E:main}
	u^\eps_t + H(Du^\eps) = F^\eps(x,t,\omega) \quad \text{in } \RR^d \times (0,\oo) \times \Omega \quad \text{and} \quad u^\eps(x,0,\omega) = u_0(x) \quad \text{in } \RR^d \times \Omega,
\end{equation}
where the initial datum $u_0$ belongs to $BUC(\RR^d)$, the space of bounded, uniformly continuous functions, and $(\Omega,\mbf F, \mbf P)$ is a given probability space.

We assume that 
\begin{equation}\label{A:introH}
	H: \RR^d \to \RR \text{ is convex with superlinear growth,}
\end{equation}
and the noise term $F^\eps$, which is scaled by a small parameter $\eps > 0$, is white in time and smooth in space:
\begin{equation}\label{scalednoise}
	\left\{
	\begin{split}
		&F^\eps(x,t,\omega) := F\pars{ \frac{x}{\eps}, \frac{t}{\eps} ,\omega}, \quad \text{where}\\
		&F(x,t,\omega) := f(x,\omega) \cdot \dot B(t) = \sum_{i=1}^m f^i(x,\omega) \dot B^i(t,\omega),\\
		&f = (f^1,f^2,\ldots,f^m) \text{ is a smooth, stationary-ergodic random field, and}\\
		&B = (B^1,B^2,\cdots,B^m): [0,\oo) \times \Omega \to \RR^m \text{ is a Brownian motion independent of $f$.}
	\end{split}
	\right.
\end{equation}
More precise assumptions will be given in Section \ref{S:A}.

\subsection{The homogenization result}

Our main goal is to demonstrate that, as $\eps \to 0$, the limiting behavior of \eqref{E:main} is governed by a deterministic, homogenized initial value problem
\begin{equation} \label{E:mainlimit}
	\oline{u}_t + \oline{H}(D \oline{u}) = 0 \quad \text{in } \RR^d \times (0,\oo) \quad \text{and} \quad \oline{u}(\cdot,0) = u_0 \quad \text{in } \RR^d.
\end{equation}

\begin{theorem}\label{T:introhomog}
	Assume \eqref{A:introH} and \eqref{scalednoise}. Then there exists a deterministic, convex, super-linear Hamiltonian $\oline{H}: \RR^d \to \RR$ such that, for all $u_0 \in BUC(\RR^d)$, the solution $u^\eps$ of \eqref{E:main} converges locally uniformly with probability one to the viscosity solution $\oline{u}$ of \eqref{E:mainlimit}.
\end{theorem}

\subsection{The enhancement effect}

The scaling properties of Brownian motion imply that, in law,
\[
	F^\eps(x,t,\omega) \stackrel{d}{=} \eps^{1/2} f \pars{ \frac{x}{\eps} ,\omega} \cdot \dot B(t,\omega),
\]
and so formally, as $\eps \to 0$, the right-hand side of \eqref{E:main} converges to zero. Nevertheless, although singular terms no longer appear in \eqref{E:mainlimit}, it turns out that the noise has a nontrivial effect on the limiting equation.

\begin{theorem}\label{T:introenhancement}
	In addition to the hypotheses of Theorem \ref{T:introhomog}, assume that $f$ is not constant on $\RR^d$. Then
	\begin{equation}\label{effectiveHestimates}
			\oline{H}(p) > H(p) \quad \text{for all } p \in \RR^d.
	\end{equation}
\end{theorem}

The particular case of the eikonal equation
\begin{equation}\label{E:eikonal}
	u^\eps_t + \frac{1}{2} |Du|^2 = F^\eps(x,t,\omega) \quad \text{in } \RR^d \times (0,\oo) \times \Omega \quad \text{and} \quad u^\eps(\cdot,0,\omega) = u_0 \quad \text{in } \RR^d \times \Omega
\end{equation}
is a simplified model for turbulent combustion, in which the evolving region
\[
	U^\eps_t := \left\{ (x,y) \in \RR^d \times \RR : y > u^\eps(x,t) \right\}
\]
and its complement represent respectively ``burnt'' and ``unburnt'' regions in a rough, dynamic environment. The noise term $F^\eps$ corresponds to random turbulence, which, according to Theorem \ref{T:introenhancement}, gives rise to an average, large-scale enhancement effect on the velocity of the interface. 

We will also investigate the effect that varying the strength of the noise has on the limiting problem. More precisely, for some $\theta \in \RR$ and for $f$ and $B$ as in \eqref{scalednoise}, we study the initial value problem
\begin{equation}\label{differentamp}
	u^\eps_t + H(Du^\eps) = \eps^\theta f\pars{ \frac{x}{\eps},\omega } \cdot \dot B\pars{t,\omega} \quad \text{in } \RR^d \times (0,\oo) \times \Omega \quad \text{and} \quad u^\eps(x,0,\omega) = u_0 \quad \text	{in } \RR^d \times \Omega.
\end{equation}
The following result explains the relationship between the size of the noise (in terms of $\theta$) and the enhancement property.
	
\begin{theorem}\label{T:introdifferentscalings}
	Under the assumptions of Theorem \ref{T:introenhancement}, let $\theta \in \RR$ and let $u^\eps$ be the solution of \eqref{differentamp}.
	
	\begin{enumerate}[(a)]
	\item If $\theta > 1/2$, then, as $\eps \to 0$, $u^\eps$ converges locally uniformly in probability to the solution $u$ of
	\[
		u_t + H(Du) = 0 \quad \text{in } \RR^d \times (0,\oo) \quad \text{and} \quad u(\cdot,0) = u_0 \quad \text{in } \RR^d.
	\]
	\item If $\theta < 1/2$, then, as $\eps \to 0$, $u^\eps$ converges locally uniformly in $\RR^d \times (0,\oo)$ in probability to $-\oo$.
	
	\item If $\theta = 1/2$, then there exists a deterministic, convex Hamiltonian $\oline{H}:\RR^d \to \RR$ with $\oline{H} > H$ such that, as $\eps \to 0$, $u^\eps$ converges locally uniformly in probability to the solution $\oline{u}$ of 
	\[
		\oline{u}_t + \oline{H}(D\oline{u}) = 0 \quad \text{in } \RR^d \times (0,\oo) \quad \text{and} \quad \oline{u}(\cdot,0) = u_0 \quad \text{in } \RR^d \times \{0\}.
	\]
	\end{enumerate}
\end{theorem}

\subsection{A regularity result}

The convergence results are proved by applying the sub-additive ergodic theorem \cite{AK} to particular solutions of the equation. A crucial tool in the analysis is a H\"older regularity estimate for solutions of
\begin{equation}\label{E:isitregular}
	u_t + H(Du) = \sum_{i=1}^m f^i(x,\omega) \dot B^i(t,\omega) \quad \text{in } \RR^d \times (0,\oo) \times \Omega
\end{equation}
that is invariant under the scaling $(x,t) \to (x/\eps,t/\eps)$.

If the Brownian motion $B$ is replaced with a continuously differentiable path, with $\dot B(t,\omega)$ bounded uniformly in $(t,\omega) \in [0,\oo) \times \Omega$, then \eqref{E:isitregular} is a Hamilton-Jacobi equations of the form
\begin{equation}\label{E:generalxt}
	u_t + \tilde H(Du,x,t) = 0 \quad \text{in } \RR^d \times (0,\oo)
\end{equation}
for some $\tilde H \in C(\RR^d \times \RR^d \times [0,\oo))$. The results of \cite{CC,C,CS,Schwab} imply that the H\"older semi-norm of $u$ can be locally controlled in terms of the growth of $\nor{f}{\oo}$, $\| \dot B \|_\oo$, $\nor{u}{\oo}$, and the growth of $H$ in $Du$. However, none of these works apply to \eqref{E:isitregular}, where the right-hand side is not only unbounded, but nowhere point-wise defined. 

The transformation
\[
	\tilde u(x,t,\cdot) := u(x,t,\cdot) - \sum_{i=1}^m f^i(x,\cdot) B^i(t,\cdot) 
\]
leads to the equation
\begin{equation}\label{E:introtransformed}
	\tilde u_t + H \pars{ D \tilde u + Df(x,\omega) \cdot B(t,\omega) } = 0 \quad \text{in } \RR^d \times (0,\oo) \times \Omega,
\end{equation}
which, for each fixed $\omega \in \Omega$, is a classical Hamilton-Jacobi equation of the form \eqref{E:generalxt}. Applying known regularity results to \eqref{E:introtransformed} then yields H\"older estimates that depend on $\nor{Df}{\oo}$, which presents a major obstacle to finding estimates for \eqref{E:generalxt} that are scale-invariant.

These issues are resolved by the following result, which is of independent interest.

\begin{theorem}\label{T:introregularity}
	Fix $M,R > 0$, and assume that $H$ satisfies \eqref{A:introH}, $f \in C^1_b(\RR^d,\RR^m)$, and $B$ is a standard $m$-dimensional Brownian motion on the probability space $(\Omega,\mbf F, \mbf P)$. Then there exist constants $C_1 = C_1(R,M) > 0$, $\alpha,\beta > 0$, and, for all $p \ge 1$, $C_2 = C_2(R,M,p) > 0$ such that, if
	\[
		\nor{f}{\oo}\cdot \nor{Df}{\oo} + \nor{f}{\oo} + \nor{u(\cdot,0)}{\oo} \le M,
	\]
	then, for all $\lambda \ge 1$,
	\[
		\mbf P \pars{ \sup_{(x,s),(y,t) \in B_R \times [1/R,R]} \frac{ \abs{ u(x,s) - u(y,t)}}{|x-y|^\alpha + |s-t|^\beta} > C_1 + \lambda } \le \frac{C_2 \nor{f}{\oo}^p}{\lambda^{p}}.
	\]
\end{theorem}

For fixed $\eps > 0$, the equation \eqref{E:main} can be rewritten in the form
\[
	u^\eps_t + H(Du^\eps) = f^\eps(x,\omega) \cdot \dot B^\eps(t,\omega) \quad \text{in } \RR^d \times (0,\oo) \times \Omega,
\]
where, for $x \in \RR^d$ and $t \in [0,\oo)$,
\[
	f^\eps(x,\cdot) = \eps^{1/2} f(x/\eps,\cdot) \quad \text{and} \quad B^\eps(t,\cdot) = \eps^{1/2} B(t/\eps,\cdot).
\]
Theorem \ref{T:introregularity} then immediately implies that, for all $\eps > 0$ and $\lambda \ge 1$,
\[
	\mbf P \pars{ \sup_{(x,s),(y,t) \in B_R \times [1/R,R]} \frac{ \abs{ u^\eps(x,s) - u^\eps(y,t)}}{|x-y|^\alpha + |s-t|^\beta} > C_1 + \lambda } \le \frac{C_2 \nor{f}{\oo}^{p} \eps^{p/2}}{\lambda^{p}},
\]
where $C_1$, $C_2$, $\alpha$, and $\beta$ are all independent of $\eps$. 

\subsection{Background}

In \cite{Seeger2,Seeger1}, the author studied general asymptotic problems for equations taking the form
\[
	du^\eps + \sum_{i=0}^m H^i(Du^\eps,x/\eps) \cdot d\zeta^{i,\eps}(t) = 0 \quad \text{in } \RR^d \times (0,\oo),
\]
where each $H^i$ satisfies a self-averaging property in the spatial variable (for example, periodic, almost-periodic, or stationary and ergodic dependence), and, for some path $\zeta \in C([0,\oo), \RR^{m+1})$,
\[
	\zeta^\eps = \pars{ \zeta^{0,\eps}, \zeta^{1,\eps}, \zeta^{2,\eps}, \ldots, \zeta^{m,\eps} } \xrightarrow{ \eps \to 0} \zeta \quad \text{locally uniformly.}
\]
The limiting equations take the form
\[
	d \oline{u} + \sum_{j=1}^M \oline{H}^j(D \oline{u}) \cdot d{\tilde \zeta}^j = 0 \quad \text{in } \RR^d \times (0,\oo),
\]
for some deterministic, spatially homogenous Hamiltonians $\oline{H}^j$, $j = 1,2,\ldots,M$, a path $\tilde \zeta \in C([0,\oo),\RR^M)$, and some $M \in \NN$ possibly large than $m+1$. The results of the present paper can be placed within this framework by setting, for $(p,y,t) \in \RR^d \times \RR^d \times [0,\oo)$, $\eps > 0$, and $i = 1,2,\ldots, m$,
\[
	H^0(p,y) = H(p), \quad \zeta^{0,\eps}(t) = \zeta^{0}(t) = t, \quad H^i(p,y) = f^i(y), \quad \zeta^{i,\eps}(t) = \eps B^i\pars{ \frac{t}{\eps}}, \quad \text{and} \quad \zeta^i(t) = 0.
\]
In this context, the fact that the limiting equation takes the form \eqref{E:mainlimit} with $\oline{H} \ne H$ can be translated as saying that each effective Hamiltonian is determined by the entire collection $(H^i)_{i=0}^m$, a phenomenon which was seen in \cite{Seeger2} for a different class of problems. A form of Theorem \ref{T:introdifferentscalings}(b) was proved in \cite{Seeger1} in the case where $\theta = 0$, with $B$ replaced with a sufficiently mild approximation $B^\eps$ converging to a Brownian motion as $\eps \to 0$.

There is a vast literature on the stochastic homogenization of Hamilton-Jacobi equations like
\begin{equation}\label{E:basicSH}
	u^\eps_t + H\pars{Du^\eps, \frac{x}{\eps} ,\omega} = 0 \quad \text{and} \quad u^\eps_t + H\pars{Du^\eps, \frac{x}{\eps}, \frac{t}{\eps} ,\omega} = 0 \quad \text{in } \RR^d \times (0,\oo)
\end{equation}
set in a stationary-ergodic environment; for qualitative and quantitative results, and many variations and extensions, see \cite{ACS,AS,ATY1, ATY2,CNS,CSo,CST,FS,G,H,JSTfront,LScorrectors,NN,RT, Schwab, S, Z}. The results of the present paper are unique in that the problem is a stochastic partial differential equation, and therefore, the time-dependent forcing term is not only unbounded, but not well-defined point-wise anywhere. 

A specific example of the equations in \eqref{E:basicSH}, and another model for turbulent combustion, is the $G$-equation, for which the level sets of the solutions evolve according to the normal velocity
\[
	1 + V \pars{ \frac{x}{\eps}, \frac{t}{\eps},\omega} \cdot n,
\]
where $V: \RR^d \times [0,\oo) \times \Omega \to \RR^d$ is a stationary-ergodic velocity field and $n \in S^{d-1}$ is the outward unit normal vector to the interface. Under the assumption that $\abs{\mbf E V } < 1$, where $\mbf E$ denotes the expectation of the random field $V$, the evolving region will, on average, expand. In fact, the authors in \cite{CNS, CSo, NN} demonstrate that, over a long time and large range, the velocity is actually enhanced, that is, it is given by $\oline{a}(n)$ for some deterministic $\oline{a} \in C( S^{d-1}, \RR_+)$ satisfying 
\[
	\oline{a}(n) \ge 1 + \mbf E V \cdot n \quad \text{for all } n \in S^{d-1}.
\]
Moreover, under further assumptions on $V$, and for ``most'' directions $n \in S^{d-1}$, the inequality is strict. This should be compared with Theorem \ref{T:introenhancement}, in which an analogous strict enhancement property is observed for all $p \in \RR^d$.

The homogenization of ``viscous'' Hamilton-Jacobi equations, or Hamilton-Jacobi-Bellman equations, like
\begin{equation}\label{E:viscSH}
	\begin{split}
	&u^\eps_t - \eps \tr\left[ A\pars{ \frac{x}{\eps},\omega} D^2 u^\eps \right] + H\pars{Du^\eps, \frac{x}{\eps} ,\omega} = 0 \quad \text{and}\\
	&u^\eps_t - \eps \tr\left[ A\pars{ \frac{x}{\eps}, \frac{t}{\eps},\omega} D^2 u^\eps \right]+ H\pars{Du^\eps, \frac{x}{\eps}, \frac{t}{\eps} ,\omega} = 0 \quad \text{in } \RR^d \times (0,\oo),
	\end{split}
\end{equation}
where $A$ is a symmetric, nonnegative matrix with stationary and ergodic dependence, has also been an active area of study, see for instance \cite{AC, ACrates,ASunbounded, AT, CScorrectors, DK, FFZ, JSTvisc, KRV, KV, LShomog,LSrevisited}. A natural next step is to study the homogenization of equations like
\begin{equation}\label{E:stochvisc}
	u^\eps_t - \eps \tr\left[ A\pars{ \frac{x}{\eps}, \frac{t}{\eps},\omega} D^2 u^\eps \right]+ H\pars{Du^\eps, \frac{x}{\eps}, \frac{t}{\eps} ,\omega} = F^\eps(x,t) \quad \text{in } \RR^d \times (0,\oo),
\end{equation}
where, as in \eqref{scalednoise}, $F^\eps$ is a stationary-ergodic forcing that is white in time. 

When $H(p,x,t) = \frac{1}{2}|p|^2$, $A$ is the identify matrix, and $F^\eps$ is a space-time white noise, \eqref{E:stochvisc} is the KPZ equation, whose well-posedness was established in the seminal works of Hairer on regularity structures \cite{Hrfirst, Hrrs} (see also the work of Gubinelli et al. \cite{GIP} on paracontrolled distributions). The KPZ equation can be related, through the Hopf-Cole transform, to the multiplicative stochastic heat equation. This connection is explored in various works \cite{DGRZ, GB, GRZ, MSZ} that study homogenization problems in which the noise term formally vanishes as $\eps \to 0$, but still has a nontrivial, enhancement-type effect on the deterministic limit. For example, in \cite{GRZ}, the limiting heat equation exhibits an effective diffusivity, and the stochastic heat equation that describes the fluctuations has an effective variance. 

A central part of understanding the macroscopic behavior of stochastically forced Hamilton-Jacobi equations concerns the existence of global solutions of \eqref{E:main} with stationary increments. In the homogenization literature, such solutions are known as correctors, and they play an important role in obtaining error estimates and determining fine properties of the effective Hamiltonian $\oline{H}$; see \cite{DavSic,LScorrectors,CScorrectors}. Additionally, they are related to the long-time behavior and invariant measures for the random dynamical system associated to the equation, as well as the asymptotic slopes of one-sided minimizers of the corresponding Lagrangian system. In this context, the so-called ``shape functional'' is exactly the effective Lagrangian $\oline{L}$, which is the convex conjugate of $\oline{H}$; see Lemma \ref{L:identifyLbar} below. The work of Bakhtin and Khanin \cite{BK}, which deals with exactly such questions, puts forward a conjectured overarching theory for the global behavior of stochastically forced Hamilton-Jacobi equations, and the related forced Burgers equation. These conjectures are motivated by a body of existing results, including, but not limited to, \cite{Bnpoisson,Bn, Bnkicks, BCK,BoK,EK,KhZ}. Within this framework, the homogenization result, Theorem \ref{T:introhomog}, provides insight into the average large scale behavior of the solutions. Moreover, the enhancement property in Theorem \ref{T:introenhancement} (see Theorem \ref{T:enhancement} for a more precise statement) gives a better understanding of the effective Hamiltonian and the related shape functional. We believe that the regularity result, Theorem \ref{T:introregularity} (see also Proposition \ref{P:Lregularity}), and the methods used to prove it, can be an important tool for other parts of this ongoing program.

\subsection{Organization} In Section \ref{S:A}, we list the main assumptions on the data and prove some preliminary results. The properties of the random Lagrangian fields, and especially their regularity, are discussed in Section \ref{S:randomL}. The main tool in this section is a decomposition method for bounding moments of stochastic integrals that are not martingales. The proof of the homogenization result and the identification of the effective Hamiltonian appear in Section \ref{S:homog}, and the enhancement property is proved in Section \ref{S:enhancement}. Finally, the results of Theorem \ref{T:introdifferentscalings} are proved in Section \ref{S:differentscalings}.

\subsection{Notation}

For a domain $U \subset \RR^M \times \RR^N$ and $\alpha,\beta \in (0,1)$, $C^{\alpha}_x C^\beta_y(U)$ is the space of functions $f = f(x,y)$, where $x \in \RR^M$ and $y \in \RR^N$, that are $\alpha$-H\"older continuous in $\RR^M$ and $\beta$-H\"older continuous in $\RR^N$, with the semi-norm
\[
	[f]_{C^{\alpha}_x C^\beta_y(U)} := \sup_{(x,y), (\tilde x,\tilde y) \in U} \frac{ \abs{ f(x,y) - f(\tilde x,\tilde y)}}{|x - \tilde x|^\alpha + |y - \tilde y|^\beta}.
\]
For $k \in \NN$, $C^k_b(U)$ is the space of functions with bounded and continuous derivatives through order $k$, and
\[
	\nor{f}{C^k} := \sum_{i=0}^k \nor{D^i f}{\oo}.
\]
If $X$ is a vector space with norm $\nor{\cdot}{X}$ and $f \in C(\RR,X)$, then, for $-\oo < a < b < \oo$, we write
\[
	\nor{f}{[a,b],X} := \max_{t \in [a,b]} \nor{f(t)}{X}.
\]
In the case where $X = \RR$, if $\alpha \in (0,1)$ and $f \in C^\alpha(\RR)$, we denote the $\alpha$-H\"older semi-norm of a function $f: \RR \to \RR$ on $[a,b] \subset \RR$ by $[f]_{\alpha,[a,b]}$.

For a function $f$ on $\RR$ and $t \in \RR$, we denote interchangeably by $f_t$ and $f(t)$ the value of $f$ at $t$, depending on notational convenience.

For $q \in (1,\oo)$, $q'$ is the conjugate exponent $q/(q-1)$. Given a set $A$, $\ind_A$ denotes the indicator function. The expectation with respect to the probability measure $\mbf P$ is denoted by $\mbf E$. When it does not cause confusion, we suppress the dependence of random variables on the parameter $\omega \in \Omega$.
	
\section{Preliminaries}\label{S:A}

\subsection{The Hamiltonian} We will always assume that
\begin{equation}\label{A:Hestimates}
	\left\{
	\begin{split}
		&H: \RR^d \to \RR \text{ is convex, and, for some $C > 1$ and $q > 1$,}\\
		&\frac{1}{C} |p|^q-C \le H(p) \le C(|p|^q +1) \quad \text{for all } p \in \RR^d.
	\end{split}
	\right.
\end{equation}
The Legendre transform of $H$
\[
	H^*(v) := \sup_{p \in \RR^d} \pars{ p \cdot v - H(p)}
\]
has analogous bounds in terms of the conjugate exponent $q' := q(q-1)^{-1}$, that is, for a possibly different constant $C > 1$, we have
\begin{equation}\label{Hstarbounds}
	\frac{1}{C} |p|^{q'}- C \le H^*(p) \le C(|p|^{q'} +1) \quad \text{for all } p \in \RR^d.
\end{equation}
We shall also impose that, for some $C > 0$,
\begin{equation}\label{DHstar}
	\abs{ H^*(p_1) - H^*(p_2) } \le C \pars{ 1 + |p_1|^{q'-1} + |p_2|^{q'-1} }|p_1 - p_2| \quad \text{for all } p_1,p_2 \in \RR^d.
\end{equation}

\subsection{The random field} 

For the random forcing term
\begin{equation}\label{randomfield}
	F(x,t,\omega) = f(x,\omega) \cdot \dot B(t,\omega) = \sum_{i=1}^m f^i(x,\omega) \dot B^i(t,\omega)
\end{equation}
and the probability measure $\mbf P$, we assume the following:
\begin{equation}\label{A:BM}
	B: [0,\oo) \times \Omega \to \RR^m \quad \text{is a standard Brownian motion},
\end{equation}
\begin{equation}\label{A:fC1}
	f(\cdot,\omega) \in C^1_b(\RR^d,\RR^m) \text{ with probability one},
\end{equation}
\begin{equation}\label{A:fBindependent}
	\text{$f$ and $B$ are independent,}
\end{equation}
and there exists a group of transformations
\[
	(\tau_x)_{x \in \RR^d} : \Omega \mapsto \Omega
\]
satisfying
\begin{equation}\label{A:group}
	\tau_{x+y} = \tau_x \circ \tau_y \quad \text{for all }x,y \in \RR^d,
\end{equation}
\begin{equation}\label{A:shifts}
	f(x,\tau_y\omega) = f(x+y,\omega) \quad \text{and} \quad B(\cdot,\tau_y \omega) = B(\cdot,\omega) \quad \text{for all } x,y \in \RR^d \text{ and } \omega \in \Omega,
\end{equation}
\begin{equation}\label{A:stationary}
	\mbf P \circ \tau_x = \mbf P \quad \text{for all } x \in \RR^d,
\end{equation}
and
\begin{equation}\label{A:muergodic}
	\text{if } A \in \sigma(f) \text{ and } \tau_x A = A \text{ for all } x \in \RR^d, \text{ then } \mbf P\pars{ A }\in \{0, 1\}.
\end{equation}
Here, $\sigma(f) \subset \mbf F$ is the $\sigma$-algebra generated by the random field $f$. 

To avoid long lists of assumptions later on, we introduce the assumption that
\begin{equation}\label{A:randomfield}
	\text{the random field $F$ satisfies \eqref{randomfield} - \eqref{A:muergodic}.}
\end{equation}

Without loss of generality, it can be assumed that $\Omega$ and $\mbf P$ have a product-like structure. That is, we may take $\Omega = X \times Y$ and $\mbf P = \mu \otimes \nu$, where $X = C^1_b(\RR^d \times \RR^m)$, $Y = C([0,\oo),\RR^m)$, $\mu$ is a measure on $X$ that is stationary and ergodic with respect to translations in $\RR^d$, and $\nu$ is the Wiener measure.

The stationarity and ergodicity of the shifts in space imply that, for some $M_0 > 0$,
\begin{equation}\label{Fboundedas}
	\mbf P \pars{ \nor{f}{C^1} \le M_0 } = 1.
\end{equation}

\subsection{Stability with respect to forcing terms} \label{S:stability}
We next address the issue of well-posedness for viscosity solutions of the initial value problems
\begin{equation}\label{E:genericequation}
	v_t + H(Dv) = \frac{\del}{\del t}\zeta(x,t) \quad \text{in } \RR^d \times (0,\oo) \quad \text{and} \quad u(\cdot,0) = u_0,
\end{equation}
where $u_0 \in BUC(\RR^d)$, $H$ satisfies \eqref{A:Hestimates}, and $\zeta \in C([0,\oo), C^1_b(\RR^d))$. In particular, $\zeta$ is not sufficiently regular for \eqref{E:genericequation} to be covered by the standard Crandall-Lions theory \cite{CL} of viscosity solutions of Hamilton-Jacobi equations. Instead, \eqref{E:genericequation} is a special case of the ``pathwise'' equations studied by Lions and Souganidis \cite{LS1,LS2,LS4,LS3,Snotes}. In the present situation, the well-posedness and stability of the equation is more straightforward than in these works, as a consequence of the additive structure of the noise and its smoothness in space.

For $u_0 \in BUC(\RR^d)$ fixed, let
\begin{equation}\label{solutionoperator}
	S_{u_0}: C^1([0,\oo),C^1_b(\RR^d)) \mapsto C(\RR^d \times [0,\oo))
\end{equation}
be the solution operator for \eqref{E:genericequation}, that is, $v = S_{u_0}(\zeta)$. 

\begin{lemma}\label{L:forcingstability}
	Assume $H$ satisfies \eqref{A:Hestimates}. Then $S_{u_0}$ extends continuously to the space $C([0,\oo),C^1_b(\RR^d))$.
\end{lemma}

\begin{proof}
	If $v$ is the solution of \eqref{E:genericequation}, then the function $\tilde v$ defined by
	\[
		\tilde v(x,t) := v(x,t) - \zeta(x,t) + \zeta(x,0)
	\]
	is a classical viscosity solution of the initial value problem
	\[
		\tilde v_t + H\pars{ D \tilde v + D_x \zeta(x,t) - D_x \zeta(x,0)} = 0 \quad \text{in } \RR^d \times (0,\oo) \quad \text{and} \quad \tilde v(\cdot,0) = u_0 \quad \text{in } \RR^d.
	\]
	The claim now follows from classical arguments from the theory of viscosity solutions.
\end{proof}

The stability for the equation can also be directly seen on the level of the representation formula for the solution operator. For $(x,y,s,t) \in \RR^d \times \RR^d \times [0,\oo) \times [0,\oo)$ with $s < t$, we define
\[
	\mcl A(x,y,s,t) := \left\{ \gamma \in W^{1,\oo}([s,t],\RR^d) : \gamma_s = x, \; \gamma_t = y \right\}
\]
and
\begin{equation}\label{preLzeta}
	L(x,y,s,t;\zeta) := \inf \left\{ \int_s^t \left[ H^*(\dot \gamma_r) + \frac{\del\zeta}{\del r}(\gamma_r,r)\right]dr : \gamma \in \mcl A(x,y,s,t) \right\},
\end{equation}
and then, as shown in \cite{Lbook}, for example,
\begin{equation}\label{pathwiseformula}
	S_{u_0}(\zeta)(x,t) = \inf_{y \in \RR^d} \pars{ u_0(y) + L(y,x,s,t;\zeta) }.
\end{equation}
Lemma \ref{L:forcingstability} then follows from Lemma \ref{L:distancestability} below, which establishes the stability of the Lagrangian with respect to the forcing term.

\begin{lemma}\label{L:distancestability}
	Assume \eqref{A:Hestimates}, let $L(\cdot,\zeta)$ be defined by \eqref{preLzeta}, and fix $(x,y,s,t) \in \RR^d \times \RR^d \times [0,\oo) \times [0,\oo)$ with $s < t$. Then the map
	\[
		C^1([0,\oo), C^1_b(\RR^d)) \ni \zeta \mapsto L(x,y,s,t,\zeta) \in \RR
	\]
	extends continuously to $\zeta \in C([0,\oo), C^1_b(\RR^d))$.
\end{lemma}

\begin{proof}
	Integrating \eqref{preLzeta} by parts yields
\begin{equation}\label{Lzeta}
	L(x,y,s,t;\zeta) = \zeta(y,t) - \zeta(x,s) + \inf \left\{ \int_s^t \left[ H^*(\dot \gamma_r) - D_x \zeta(\gamma_r,r)\cdot \dot \gamma_r \right]dr : \gamma \in \mcl A(x,y,s,t) \right\}.
\end{equation}
The result now follows from classical arguments, in view of the super-linearity of $H^*$ and the continuity of $D_x\zeta$ in both variables.
\end{proof}

\section{Properties of the random Lagrangian}\label{S:randomL}

We continue the discussion of the Lagrangians defined at the end of Section \ref{S:A}, and we investigate the case where the forcing term is random and given by
\[
	\zeta(x,t,\omega) := f(x) \cdot B(t,\omega) = \sum_{i=1}^m f^i(x) B^i(t,\omega) \quad \text{for } (x,t,\omega) \in \RR^d \times [0,\oo) \times \Omega,
\]
where $f \in C^1_b(\RR^d)$ and $B$ is a standard Brownian motion on the probability space $(\Omega,\mbf F, \mbf P)$. Throughout this section, $f$ is non-random, and the estimates will be uniform over certain bounded sets of $C^1_b(\RR^d)$.

We obtain uniform growth bounds and regularity estimates for the random Lagrangian
\begin{equation}\label{Lf}
	L_f(x,y,s,t,\omega) := L(x,y,s,t; f \cdot B(\omega)) = \inf\left\{ \int_s^t H^*(\dot \gamma_r)dr + \int_s^t f(\gamma_r)\cdot dB_r(\omega) : \gamma \in \mcl A(x,y,s,t) \right\}.
\end{equation}
In view of Lemma \ref{L:distancestability}, $L_f$ is well-defined for any continuous sample path $B(\cdot,\omega)$.

\subsection{Integrating non-adapted paths against Brownian motion}

The methods used below resemble those of \cite{C, Schwab, KV}, which involve the manipulation of almost-minimizers of the Lagrangian action. The new difficulty is to find a way to control, for an arbitrary Lipschitz process $\gamma$, integrals of the form
\begin{equation}\label{badintegral}
	\int_{r_1}^{r_2} f(\gamma_r) \cdot dB_r.
\end{equation}
If $\gamma$ is adapted with respect to the natural filtration of the Brownian motion, then standard It\^o calculus implies that the moments of \eqref{badintegral} can be bounded in terms of $\nor{f}{\oo}$, independently of the regularity of $f$ or $\gamma$. Therefore, as a consequence of Kolmogorov's continuity criterion (see \cite{SV}), for any $\alpha \in (0,1/2)$, \eqref{badintegral} is $\alpha$-H\"older continuous in $r_1$ and $r_2$, and, for all $p \ge 1$, $T > 0$, $\lambda \ge 1$, and some $C = C(p,\alpha,T) > 0$,
\begin{equation}\label{desiredtailbound}
	\mbf P \pars{ \sup_{r_1,r_2 \in [0,T]} \frac{1}{|r_1-r_2|^\alpha}\abs{ \int_{r_1}^{r_2} f(\gamma_r) \cdot dB_r} > \lambda} \le \frac{C\nor{f}{\oo}^p}{\lambda^p}.
\end{equation}
However, if $\gamma$ is not adapted, as is the case for the almost-minimizers in general, then It\^o calculus does not apply, and the regularity of $f$ and $\gamma$ enter into the moment estimates.

In order to control \eqref{badintegral} in such a way that allows us to obtain scale-invariant estimates for $L_f$, we decompose the integral into three parts: one that can be bounded by a deterministic constant, one which measures the regularity of $\gamma$, and a final random piece whose probability tails satisfy bounds resembling \eqref{desiredtailbound}. Crucially, the various constants depend on $\nor{f}{C^1}$ only through an upper bound for the product $\nor{f}{\oo} \cdot \nor{Df}{\oo}$.

For $\sigma > 0$ and $\Lambda > 0$, define
\begin{equation}\label{C1subspace}
	F_{\sigma,\Lambda} := \left\{ f \in C^1_b(\RR^d,\RR^m) : \nor{f}{\oo} \le \sigma \text{ and } \nor{Df}{\oo} \le \Lambda \right\}.
\end{equation}

\begin{lemma}\label{L:fauxIto}
	Fix $T > 0$, $\sigma_0 > 0$, $K > 0$, $q > 1$, and $\alpha \in (0,1/2)$. Then there exists a constant $M = M(T,\sigma_0, K,q,\alpha) > 0$ and, for every $\sigma \in (0,\sigma_0)$ and $\Lambda > 0$ satisfying $\sigma \Lambda = K$, a random variable 
	\[
		\mcl{D}_{\sigma,\Lambda}: \Omega \to \RR_+
	\]
	such that
	\begin{enumerate}[(a)]
	\item for any $p \ge 1$ and some constant $C = C(T,\sigma_0,K,p,q,\alpha) > 0$,
	\[
		\mbf P( \mcl D_{\sigma,\Lambda} \ge \lambda) \le \frac{C \sigma^p}{\lambda^p} \quad \text{for all } \lambda \ge 1,
	\]
	and
	\item for all $\gamma \in W^{1,\oo}([0,T],\RR^d)$, $f \in F_{\sigma,\Lambda}$, $\delta \in (0,1)$, and $0 \le s \le r_1 \le r_2 \le t \le T$, 
	\[
		\abs{ \int_{r_1}^{r_2} f(\gamma_r) \cdot dB_r } \le \pars{ M\delta^{q'}\int_{s}^{t} |\dot \gamma_r|^{q'}dr + \frac{M + \mcl D_{\sigma,\Lambda}}{\delta^{q} } }(r_2-r_1)^\alpha.
	\]
	\end{enumerate}
\end{lemma}

We note that there is a different random variable $\mcl D_{\sigma,\Lambda,T,\sigma_0,q,\alpha}$ corresponding to each choice of $(\sigma,\Lambda,T,\sigma_0,q,\alpha)$ satisfying the hypotheses of Lemma \ref{L:fauxIto}. However, we suppress the dependence on all but $\sigma$ and $\Lambda$, as the dependence of $\mcl D_{\sigma,\Lambda}$ on the other parameters will not play a role. The important aspect of the random variable $\mcl D_{\sigma,\Lambda}$ is not its exact formula (see \eqref{theD} below), but rather, the nature of the estimates it satisfies as claimed by the above lemma, and the fact that it is does not depend on the Lipschitz path $\gamma$ or $f \in F_{\sigma,\Lambda}$.

In order to prove Lemma \ref{L:fauxIto}, we will need a parameter-dependent generalization of the classical Kolmogorov continuity criterion. The proof follows a similar method, but we present it here for the sake of completeness.

\begin{lemma}\label{L:Kolmogorov}
	For some parameter set $\mcl M$, let $(M_\mu)_{\mu \in \mcl M}: \Omega \to \RR_+$ and $(Z_\mu)_{\mu \in \mcl M}: [0,T] \times \Omega \to \RR$ be such that, for some positive constants $m > 0$, $\beta \in (0,1)$, and $p \ge 1$,
	\[
		\sup_{0 \le s < t \le T} \mbf E \left[ \sup_{\mu \in \mcl M}\pars{  \frac{ |Z_\mu(t) - Z_\mu(s)|}{(t-s)^{\beta+1/p}} - M_\mu}_+^p \right] \le m.
	\]
	Then, for all $\alpha \in (0,\beta)$, there exist constants $C_1 = C_1(\alpha,T)$ and $C_2 = C_2(p,\alpha,\beta,T) > 0$ such that, for all $\lambda \ge 1$,
	\[
		\mbf P\pars{ \sup_{\mu \in \mcl M} \pars{ [Z_\mu]_{\alpha,[0,T]} - C_1 M_\mu} > \lambda } \le \frac{C_2 m}{\lambda^p}.
	\]
\end{lemma}

\begin{proof}
	Without loss of generality we take $T = 1$. For $n = 0,1,2,\ldots$, define 
	\[
		D_n := \left\{ \frac{k}{2^n} : k = 0,1,2,\ldots, 2^n \right\} \quad \text{and} \quad D := \bigcup_{n = 0}^\oo  D_n.
	\]
	For some constant $A = A(\alpha)> 0$ to be determined and for each $n = 0,1,2,\ldots$, define the event
	\[
		\mcl A_n := \left\{ |Z_\mu(s) - Z_\mu(t)| \le (M_\mu+A\lambda)|s-t|^\alpha \text{ for all $\mu \in \mcl M$ and $\{s,t\} \subset D_n$ such that $|s-t| = 2^{-n}$} \right\}.
	\]
	
	Fix a pair $\{s,t\}$ in $D_n$ satisfying $|s-t| = 2^{-n}$. Then
	\begin{align*}
		\mbf P &\pars{ |Z_\mu(s) - Z_\mu(t)| > (M_\mu+A\lambda) |s-t|^\alpha \text{ for some } \mu \in \mcl M}\\
		&\le \mbf P \pars{ \sup_{\mu \in \mcl M} \pars{ \frac{|Z_\mu(s) - Z_\mu(t)|}{|s-t|^{\beta + 1/p} } - M_\mu}^p_+ > (A\lambda)^p 2^{n(1+ (\beta - \alpha) p)}  }\\
		&\le mA^{-p} \lambda^{-p} 2^{-n(1 + (\beta- \alpha) p)},
	\end{align*}
	and therefore,
	\[
		\mbf P \pars{ \Omega \backslash \mcl A_n} \le mA^{-p} \lambda^{-p} 2^{-np(\beta - \alpha)}.
	\]
	
	Now, fix $\omega \in \bigcap_{n=0}^\oo \mcl A_n$ and $s,t \in D$, assume without loss of generality that $s < t$, and let $n \in \NN$ be such that
	\[
		2^{-n-1} < t-s \le 2^{-n}.
	\] 
	Then, for some $M_1,M_2 \in \NN$ and
	\[
		\pars{ s^i}_{i=0}^{M_1}, \pars{ t^j}_{j=0}^{M_2} \subset D,
	\]
	we can write
	\[
		s = s^{M_1} > s^{M_1 - 1} > s^{M_1 - 2} > \cdots > s^1 > s^0 \quad \text{and}
		\quad 
		t = t^{M_2} < t^{M_2 - 1} < t^{M_2 - 2} < \cdots < t^1 < t^0,
	\]
	where
	\[
		\left\{
		\begin{split}
			&s^i \in D_{n+i} \text{ for } i = 0,1,2,\ldots, M_1, \quad t^j \in D_{n+j} \text{ for } j = 0,1,2,\ldots,M_2,\\
			&s^i - s^{i-1} = 2^{-n-i}, \; t^{j-1} - t^j = 2^{-n-j}, \text{ and}\\
			&t_0 - s_0 = 2^{-n}.
		\end{split}
		\right.
	\]
	It follows that, for all $\mu \in \mcl M$,
	\begin{align*}
		|Z_\mu(s) - Z_\mu(t)| &\le |Z_\mu(s^0) - Z_\mu(t^0)| + \sum_{i=1}^{M_1} |Z_\mu(s^i) - Z_\mu(s^{i-1})| + \sum_{j=1}^{M_2} |Z_\mu(y^j) - Z_\mu(y^{j-1})| \\
		&\le (M_\mu+A\lambda)\pars{ 2^{-n\alpha} +  \sum_{i=1}^{M_1} 2^{-(n+i) \alpha} + \sum_{j=1}^{M_2} 2^{-(n+j) \alpha} }\\
		&\le (M_\mu+A\lambda) 2^\alpha \pars{ 1 + \sum_{i=1}^{M_1} 2^{-i\alpha} + \sum_{j=1}^{M_2} 2^{-j\alpha} }|t-s|^\alpha\\
		&\le 3 \cdot 2^\alpha (M_\mu+A\lambda)|t-s|^\alpha.	
	\end{align*}
	Choosing now $A(\alpha) = 3^{-1} 2^{-\alpha}$, we conclude that, if $\omega \in \bigcap_{n=0}^\oo \mcl A_n$, then, for some constant $C_1$ as in the statement of the lemma,
	\[
		\sup_{\mu \in \mcl M}\pars{  \sup_{s,t \in [0,1]} \frac{|Z_\mu(s) - Z_\mu(t)|}{|t-s|^\alpha} - C_1 M_\mu }\le \lambda.
	\]
	Therefore, for some $C_2 = C_2(p,\alpha,\beta,T) > 0$,
	\[
		\mbf P\pars{ \sup_{\mu \in \mcl M} \pars{[Z_\mu]_{\alpha,[0,1]}- C_1 M_\mu} >  \lambda} \le \sum_{n=0}^\oo \mbf P(\Omega \backslash \mcl A_n) \le m A(\alpha)^{-p} \lambda^{-p}\sum_{n=0}^\oo 2^{-np(\beta - \alpha)} \le C_2 m\lambda^{-p}.
	\]
\end{proof}

\begin{proof}[Proof of Lemma \ref{L:fauxIto}] We define the parameter set
	\[
		\mcl M := (0,1) \times \{ (s,t) : 0 \le s \le t \le T \} \times W^{1,\oo}([0,T],\RR^d) \times F_{\sigma,\Lambda},
	\]
	where $F_{\sigma,\Lambda}$ is the subspace of $C^1(\RR^d,\RR^m)$ defined in \eqref{C1subspace}, and, for each $\mu = (\delta,(s,t),\gamma,f) \in \mcl M$ and $u \in [0,T]$, the stochastic process
	\begin{equation}\label{subadditiveguy}
		Z_\mu(u) := \delta^{q} \int_0^u f(\gamma_r) \ind_{[s,t]}(r) \cdot dB_r.
 	\end{equation}
	We first show that there exists a constant $M = M(T,\sigma_0,K,q) > 0$ and, for all $p \ge 1$, a constant $C = C(T,\sigma_0,K,p,q) > 0$ such that
	\begin{equation}\label{pointwisemoment}
		\sup_{0 \le r_1 \le r_2 \le T} \mbf E \left[ \sup_{\mu \in \mcl M} \pars{  \frac{ |Z_\mu(r_2) - Z_\mu(r_1)|}{ (r_2 - r_1)^{1/2} } - M \delta^{q} \pars{ \delta^{q'} \int_s^t |\dot \gamma_r|^{q'} dr + \frac{1}{\delta^{q} } }  }^p_+ \right] \le C \sigma^p.
	\end{equation}
	
	Fix $r_1,r_2 \in [0,T]$ with $r_1 \le r_2$. Then
	\[
		Z_\mu(r_1) - Z_\mu(r_2) = \delta^q\int_{r_1}^{r_2} f(\gamma_r)\ind_{[s,t]}(r) \cdot dB_r = \delta^q\int_{r_1 \vee s}^{r_2 \wedge t} f(\gamma_r) \cdot dB_r.
	\]
	We now split into several cases, depending on the relative sizes and positions of the intervals $[r_1,r_2]$ and $[s,t]$.
	
	{\it Case 1.} Assume first that
\begin{equation}\label{Deltacase1}
	r_2 - r_1 \le \frac{\sigma^q}{\Lambda^q}.
\end{equation}

Integrating by parts, we have
\[
	\int_{r_1}^{r_2} f(\gamma_r)\ind_{[s,t]}(r) \cdot dB_r = f(\gamma_{r_2 \wedge t})\cdot (B_{r_2 \wedge t} - B_{r_1 \vee s}) + \int_{r_1\vee s}^{r_2 \wedge t} Df(\gamma_r) \dot \gamma_r \cdot (B_r - B_{r_1 \vee s})dr.
\]
Set
\[
	X_0 := \max_{u,v \in [r_1,r_2]} \abs{ B_u - B_v}.
\]
Young's inequality and \eqref{Deltacase1} then give, for some constant $C = C(\sigma_0, q) > 0$,
\begin{align*}
	\abs{ \int_{r_1}^{r_2} f(\gamma_r)\ind_{[s,t]}(r) \cdot dB_r} &\le \sigma X_0  + \Lambda X_0 \int_{r_1 \vee s}^{r_2 \wedge t} |\dot \gamma_r|dr\\
	&\le \sigma \pars{ X_0 + X_0\pars{\int_{s}^{t} |\dot \gamma_r|^{q'}dr}^{1/q'}}\\
	&\le \sigma \pars{ X_0+   C\frac{X_0^q }{ \delta^q (r_2 - r_1)^{(q-1)/2}} } + (r_2-r_1)^{1/2} \delta^{q'} \int_{s}^{t} |\dot \gamma_r|^{q'}dr .
\end{align*}
Therefore, 
\[
	\sup_{\mu \in \mcl M} \pars{ |Z_\mu(r_2) - Z_\mu(r_1)| - \delta^{q' + q} \int_{s}^{t} |\dot \gamma_r|^{q'}dr (r_2 - r_1)^{1/2}}
	\le \sigma \pars{ X_0 +  \frac{C X_0^q}{(r_2 - r_1)^{(q-1)/2} } },
\]
and \eqref{pointwisemoment} holds in this case for any $M > 0$, in view of the fact that, for any $m > 0$, there exists a constant $c(m) > 0$ such that
\[
	\mbf E X_0^m \le c (r_2 - r_1)^{m/2}.
\]

{\it Case 2.} Assume now that
\begin{equation}\label{Deltacase2}
	r_2 - r_1 > \frac{\sigma^q}{\Lambda^q}.
\end{equation}
Set
\begin{equation}\label{meshsizeh}
	h := \frac{\sigma}{\Lambda} (r_2 - r_1)^{1/q'},
\end{equation}
and let $N \in \NN$ be such that
\[
	\frac{r_2 - r_1}{h} \le N < \frac{r_2 - r_1}{h} + 1.
\]
Note that \eqref{Deltacase2} implies that $Nh$ is proportional to the size of the interval $[r_1,r_2]$, and, in particular,
\[
	r_2 - r_1 \le Nh < 2(r_2 - r_1).
\]

For $k = 0,1,2,\ldots, N-1$, set $\tau_k := r_1 + kh$ and $\tau_N = r_2$, and, for $k = 1,2,\ldots, N$, define
\[
	X_k = \max_{u,v \in [\tau_{k-1},\tau_k]} \abs{ B_u - B_v}.
\]

We claim that, for all $\eps > 0$, $0 \le s \le t \le T$, $\gamma \in \mcl A$, and $f \in F_{\sigma, \Lambda}$,
\begin{equation}\label{unifst}
	\abs{ \int_{r_1}^{r_2} f(\gamma_r) \ind_{[s,t]}(r) \cdot dB_r } \le \sigma \sum_{k=1}^N X_k + \Lambda h^{1/q} \pars{ \frac{1}{q\eps^q} \sum_{k=1}^N X_k^q +  \frac{\eps^{q'}}{q'} \int_{s}^{t} |\dot \gamma_r|^{q'}dr }.
\end{equation}
The proof of \eqref{unifst} will depend on whether $t-s$ is small or large compared to $r_2 - r_1$.

Choose $m, n \in \NN$ such that
\[
	\tau_{m-1} < s \le \tau_m \quad \text{and} \quad \tau_n \le t < \tau_{n+1},
\]
and, in what follows, define $\tau_{-1} := -\oo$, $\tau_{N+1} := +\oo$, and $X_0 = X_{N+1} = 0$ for consistency. Observe that $n \ge m-1$.

{\it Case 2a.}  If $n = m-1$, that is,
\[
	\tau_{m-1} < s \le t < \tau_m,
\]
then
\[
	\int_{r_1}^{r_2} f(\gamma_r)\ind_{[s,t]}(r) \cdot dB_r = f(\gamma_t)\cdot (B_t - B_s) - \int_s^t Df(\gamma_r) \dot \gamma_r  \cdot (B_r - B_s)dr
\]
and so Young's inequality and the fact that $t-s \le h$ yield
\begin{align*}
	\abs{\int_{r_1}^{r_2} f(\gamma_r)\ind_{[s,t]}(r) \cdot dB_r } &\le \sigma X_m + \Lambda X_m \int_s^t |\dot \gamma_r|dr\\
	&\le \sigma X_m + \Lambda h^{1/q} X_m \pars{ \int_s^t |\dot \gamma_r|^{q'} dr}^{1/q'}\\
	&\le \sigma X_m + \Lambda h^{1/q} \pars{ \frac{1}{q\eps^q} X_m^q + \frac{\eps^{q'}}{q'} \int_s^t |\dot \gamma_r|^{q'}dr }.
\end{align*}
Therefore, \eqref{unifst} holds.

{\it Case 2b.} Assume now that $n > m-1$. Then
\begin{align*}
	\int_{r_1}^{r_2} f(\gamma_r)\ind_{[s,t]}(r) \cdot dB_r
	&= \sum_{k=m}^{n+1} \int_{\tau_{k-1}}^{\tau_k} f(\gamma_r) \ind_{[r_1 \vee s, r_2 \wedge t]}(r) \cdot dB_r \\
	&= \mathrm{I} - \mathrm{II},
\end{align*}
where
\[
	\mathrm{I} := f(\gamma_{\tau_m}) \cdot (B_{\tau_m} - B_{s \vee r_1}) + f(\gamma_{t \wedge r_2}) \cdot (B_{t \wedge r_2} - B_{\tau_n} ) + \sum_{k=m+1}^n f(\gamma_{\tau_k}) \cdot (B_{\tau_k} - B_{\tau_{k-1}} )
\]
and
\begin{align*}
	\mathrm{II} &:= \int_{s \vee r_1}^{\tau_m} Df(\gamma_r)\dot \gamma_r \cdot (B_r - B_{s \vee r_1} )dr
	+ \int_{\tau_n}^{t \wedge r_2} Df(\gamma_r) \dot \gamma_r \cdot (B_r - B_{\tau_n})dr\\
	&+ \sum_{k=m+1}^n \int_{\tau_{k-1}}^{\tau_k} Df(\gamma_r)\dot \gamma_r \cdot (B_r - B_{\tau_{k-1}})dr.
\end{align*}
The inequality \eqref{unifst} is then a consequence of the estimates
\[
	|\mathrm{I} | \le \sigma \sum_{k=m}^{n+1} X_k \le \sigma \sum_{k=1}^N X_k
\]
and
\begin{align*}
	|\mathrm{II}| &\le \Lambda \pars{ X_m \int_{s \vee r_1}^{\tau_m} |\dot \gamma_r|dr + X_{n+1} \int_{\tau_n}^{t \wedge r_2} |\dot \gamma_r|dr + \sum_{k=m+1}^{n} X_k \int_{\tau_{k-1}}^{\tau_k} |\dot \gamma_r|dr}\\
	&\le \Lambda h^{1/q} \pars{ X_m \pars{ \int_{s}^{\tau_m} |\dot \gamma_r|^{q'}dr}^{1/q'}+ X_{n+1} \pars{ \int_{\tau_n}^{t} |\dot \gamma_r|^{q'} dr}^{1/q'} + \sum_{k=m+1}^{n} X_k \pars{\int_{\tau_{k-1}}^{\tau_k} |\dot \gamma_r|^{q'} dr}^{1/q'} }\\
	&\le \Lambda h^{1/q} \pars{ \frac{1}{q\eps^q} \sum_{k=m}^{n+1} X_k^q +  \frac{\eps^{q'}}{q'} \int_{s}^{t} |\dot \gamma_r|^{q'}dr }\\
	&\le \Lambda h^{1/q} \pars{ \frac{1}{q\eps^q} \sum_{k=1}^{N} X_k^q +  \frac{\eps^{q'}}{q'} \int_{s}^{t} |\dot \gamma_r|^{q'}dr }.
\end{align*}

We now set
\[
	\eps := \delta\pars{ N^{1/q} h^{1/2} }^{1/q'},
\]
so that \eqref{unifst} becomes
\[
	\abs{ \int_{r_1}^{r_2} f(\gamma_r)\ind_{[s,t]}(r) \cdot dB_r} \le \sigma \sum_{k=1}^N X_k + \Lambda h^{1/q} \pars{ \frac{ (N^{1/q} h^{1/2})^{1-q} }{q\delta^q} \sum_{k=1}^N X_k^q + \frac{\delta^{q'} N^{1/q} h^{1/2} }{q'} \int_{s}^{t} |\dot \gamma_r|^{q'}dr }.
\]
For $k = 1,2,\ldots,N$, the constants
\[
	a_k := \mbf E X_k \quad \text{and} \quad b_k := \mbf E X_k^q,
\]
satisfy, for some $a > 0$ and $b = b(q) > 0$,
\[
	a_k \le ah^{1/2} \quad \text{and} \quad b_k \le b h^{q/2},
\]
and so, using the definition of $h$ in \eqref{meshsizeh},
\[
	\sigma \sum_{k=1}^N a_k \le \sigma Nh^{1/2} \le 2a\sigma (r_2 - r_1) h^{-1/2} = 2a K^{1/2} (r_2 - r_1)^{\frac{q+1}{2q}}
\]
and
\begin{align*}
	\Lambda h^{1/q} \cdot (N^{1/q} h^{1/2})^{1-q} \sum_{k=1}^N b_k 
	&\le b \Lambda (Nh)^{1/q} h^{1/2}\\
	&\le 2^{1/q} b \Lambda (r_2 - r_1)^{1/2} h^{1/2} \\
	&= 2^{1/q} b K^{1/2} (r_2 - r_1)^{\frac{q+1}{2q}}.
\end{align*}
Similarly,
\[
	\Lambda h^{1/q} N^{1/q} h^{1/2} \le 2^{1/q} \Lambda (r_2 - r_1)^{1/q} h^{1/2} = 2^{1/q} K^{1/2} (r_2 - r_1)^{\frac{q+1}{2q}},
\]
and therefore, for some constant $M = M(T,K,q) > 0$,
\begin{align*}
	\sup_{\mu \in \mcl M} &\pars{ |Z_\mu(r_2) - Z_\mu(r_1)|- M \delta^q \pars{ \delta^{q'} \int_s^t |\dot \gamma_r|^{q'} dr + \frac{1}{\delta^q} }(r_2 - r_1)^{1/2}  }_+ \\
	&\le M \pars{\sigma \abs{ \sum_{k=1}^N \pars{ X_k - a_k} } + \Lambda (r_2 - r_1)^{1/q}  h^{1/2} \cdot \frac{1}{ Nh^{q/2}}\abs{ \sum_{k=1}^N (X_k^q- b_k)}}.
\end{align*}
The collections $(X_k - a_k)_{k=1}^N$ and $(X_k^q  -b_k)_{k=1}^N$ consist of independent, mean-zero random variables that satisfy, for any $k = 1,2,\ldots, N$, $m \ge 1$, and some constants $a' = a'(m) > 0$ and $b' = b'(q,m)$,
\[
	\mbf E \abs{ X_k - a_k}^m \le a' h^{m/2} \quad \text{and} \quad \mbf E \abs{ X_k^q - b_k}^m \le b' h^{mq/2}.
\]
Therefore, for each $p \ge 1$, there exist constants $A = A(p) > 0$ and $B = B(p,q) > 0$ such that
\[
	\mbf E \abs{ \sum_{k=1}^N \pars{ X_k - a_k} }^p \le A N^{p/2} h^{p/2}
\]
and
\[
	\mbf E \abs{ \sum_{k=1}^N (X_k^q- b_k)}^p \le B N^{p/2}  h^{pq/2}.
\]
Combining all terms in the inequality leads to \eqref{pointwisemoment}.

We now take $p$ large enough that
\[
	\alpha < \frac{1}{2} - \frac{1}{p}.
\]
Then \eqref{pointwisemoment} and Lemma \ref{L:Kolmogorov} imply that, for some $C = C(T,\sigma_0,K,p,q,\alpha) > 0$ and $M = M(T,\sigma_0, K, q,\alpha) > 0$, possibly different from above, and for all $\lambda \ge 1$ and
\[
	p > \frac{2}{1-2\alpha},
\]
we have
\[
	\mbf P \pars{ \sup_{\mu \in \mcl M}\pars{ [Z_\mu]_{\alpha,[0,T]} - M\delta^q\pars{ \delta^{q'}  \int_s^t |\dot \gamma_r|^{q'}dr + \frac{1}{\delta^{q}} } } > \lambda} \le \frac{C \sigma^p}{\lambda^p}. 
\]
By changing the value of $C$, if necessary, depending only on $\sigma_0$, the same can be accomplished for any $p \ge 1$. The proof of Lemma \ref{L:fauxIto} is then finished upon defining
\begin{equation}\label{theD}
	\mcl D_{\sigma,\Lambda} :=  \sup_{\mu \in \mcl M}\pars{ [Z_\mu]_{\alpha,[0,T]} - M\delta^q\pars{ \delta^{q'}  \int_s^t |\dot \gamma_r|^{q'}dr + \frac{1}{\delta^{q}}  } }.
\end{equation}
\end{proof}

We now use Lemma \ref{L:fauxIto} to estimate the growth of the Lagrangian, and the $W^{1,q'}$-norms for almost-minimizers of $L_f$, in terms of $\mcl D_{\sigma,\Lambda}$. Note that the exponent $q'$ in Lemma \ref{L:fauxIto} is chosen so as to match the growth of $H^*$ in the definition of $L_f$.

\begin{lemma}\label{L:minimizer}
	Assume that $H$ satisfies \eqref{A:Hestimates}. Then there exists $C = C(T,\sigma_0, K, q,\alpha) > 0$ such that, if $\sigma \in (0,\sigma_0)$, $\sigma \Lambda = K$, $f \in F_{\sigma,\Lambda}$, and $(x,y,s,t) \in \RR^d \times \RR^d \times [0,T] \times [0,T]$ with $s < t$, then
	\begin{equation}\label{randomLbounds}
		-C(1+\mcl D_{\sigma,\Lambda})(t-s)^{\alpha} + \frac{1}{C} \frac{|y-x|^{q'}}{(t-s)^{q'-1}} \le L_f(x,y,s,t) \le C(1+\mcl D_{\sigma,\Lambda})(t-s)^{\alpha} + C \frac{|y-x|^{q'}}{(t-s)^{q'-1}},
	\end{equation}
	and, if $\gamma: [s,t] \times \Omega \to \RR^d$ is such that $\gamma \in \mcl A(x,y,s,t)$ and
	\begin{equation}\label{almostminimizer}
		L_f(x,y,s,t) + 1 \ge \int_s^t H^*(\dot \gamma_r)dr + \int_s^t f(\gamma_r) \cdot dB_r \quad \text{with probability one,}
	\end{equation}
	then
	\begin{equation}\label{minLq}
		\int_s^t |\dot \gamma_r|^{q'}dr \le C \pars{ \frac{|y-x|^{q'}}{(t-s)^{q'-1} } + \mcl D_{\sigma,\Lambda} + 1 }.
	\end{equation}
\end{lemma}

\begin{proof}
We apply the Lagrangian action to the linear path
\[
	\ell := \pars{ x + \frac{y-x}{t-s}(r-s)}_{r \in [s,t]} \in \mcl A(x,y,s,t),
\]
and, appealing to Lemma \ref{L:fauxIto} with $\delta = 1$, we find
\[
	L_f(x,y,s,t) \le \int_s^t H^*(\dot \ell_r)dr + \int_s^t f(\ell_r) \cdot dB_r \le C \pars{ \frac{|y-x|^{q'}}{(t-s)^{q'-1}} + (1+\mcl D_{\sigma,\Lambda})(t-s)^\alpha}.
\]
This establishes the upper bound of \eqref{randomLbounds}.

Next, \eqref{Hstarbounds} and Lemma \ref{L:fauxIto} yield, for any $\delta \in (0,1)$ and $\gamma \in \mcl A(x,y,s,t)$,
\begin{equation}\label{firstuseofdelta}
	\int_s^t H^*(\dot \gamma_r)dr + \int_s^t f(\gamma_r) \cdot dB_r \ge \pars{ \frac{1}{C} - M\delta^{q'}} \int_s^t |\dot \gamma_r|^{q'}dr - \frac{1}{\delta^q} (M + \mcl D_{\sigma,\Lambda} )(t-s)^\alpha.
\end{equation}
Taking $\delta$ sufficiently small, employing Jensen's inequality, and taking the infimum over $\gamma \in \mcl A(x,y,s,t)$ gives the lower bound in \eqref{randomLbounds}. Finally, \eqref{minLq} is a consequence of \eqref{almostminimizer}, \eqref{firstuseofdelta}, and the upper bound in \eqref{randomLbounds}.
\end{proof}

\subsection{The regularity estimate}
For $R > 0$ and $0 < \tau < T$, define the domain
\[
	U_{\tau,R,T} := \left\{ (x,y,s,t) \in \RR^d \times \RR^d \times [0,T] \times [0,T] : |x-y| \le R \text{ and } \tau < t - s < T \right\}.
\]

\begin{proposition}\label{P:Lregularity}
	Assume $H$ satisfies \eqref{A:Hestimates}, and let $\sigma_0 > 0$, $K > 0$, and $0 < \theta < \frac{q}{2+q}$. Then there exist $M = M(R,T,\tau,\sigma_0, K,\theta,q) > 0$ and, for any $p \ge 1$, $C = C(T,\sigma_0, K,\theta,p,q) > 0$ such that, if $\sigma \in (0,\sigma_0)$ and $\sigma\Lambda = K$, then
	\[
		\mbf P \pars{\sup_{f \in F_{\sigma,\Lambda} } [L_f]_{C_{x,y}^{\theta}C_{s,t}^{\theta/q}(U_{\tau,R,T}) } > M + \lambda} \le \frac{CM^p\sigma^p}{\lambda^p} \quad \text{for all } \lambda \ge 1.
	\]
\end{proposition}

\begin{proof}
	We choose $\alpha \in (0,1/2)$ from Lemma \ref{L:fauxIto} sufficiently close to $1/2$ that
	\begin{equation}\label{thetaalpha}
		\theta = \frac{\alpha q}{1 + \alpha q}.
	\end{equation}
	In what follows, the constants $C > 0$ and $M > 0$ depend on $R$, $T$, $\tau$, $\sigma_0$, $K$, $q$, and $\theta$ (and therefore $\alpha$) unless otherwise specified, and may change from line to line.
	
	{\it Step 1: Spatial increments.} Fix $(x,y,s,t), (x,\tilde y, s,t) \in U_{R,T,\tau}$ and $\nu \in (0,1)$, and let $\gamma \in \mcl A(x,y,s,t)$ satisfy
	\begin{equation}\label{numin}
		L_f(x,y,s,t) + \nu \ge \int_s^t H^*(\dot \gamma_r)dr + \int_s^t f(\gamma_r) \cdot dB_r.
	\end{equation}
	Set
	\[
		h := c (t-s)|\tilde y -y|^{q/(1+\alpha q)},
	\] 
	where $c = c(q,\alpha,R)$ is chosen so that $s < t - h < t$, and define $\tilde \gamma \in \mcl A(x,\tilde y,s,t)$ by
	\[
		\tilde \gamma_r := 
		\begin{dcases}
			\gamma_r & \text{for } r \in [s,t - h), \text{ and}\\
			\gamma_r + \frac{\tilde y - y}{h}(r - t + h) & \text{for } r \in [t - h,t].
		\end{dcases}
	\]
	Observe that
	\begin{equation}\label{comparativepaths}
		\nor{\dot{\tilde \gamma}}{q'} \le \nor{\dot\gamma}{q'} + \pars{ \int_{t - h}^t \abs{ \frac{\tilde y- y}{h} }^{q'}dr}^{1/q'} = \nor{\dot\gamma}{q'} + \frac{|\tilde y- y|}{h^{1/q}} \le \nor{\dot \gamma}{q'} +  C |\tilde y- y|^{\alpha q/(1+ \alpha q)} \le \nor{\dot \gamma}{q'} + C,
	\end{equation}
	and so, by Lemma \ref{L:minimizer},
	\[
		\int_s^t \pars{ |\dot \gamma_r|^{q'} + |\dot{\tilde \gamma}_r |^{q'} } dr \le C( 1 + \mcl D_{\sigma,\Lambda}).
	\]
	
	The definition of $L_f$ then yields
	\[
		L_f(x,\tilde y,s,t) - L_f(x,y,s,t) - \nu \le \int_{t-h}^t \left[ H^*\pars{ \dot \gamma_r + \frac{\tilde y - y}{h}}  - H^*\pars{ \dot \gamma_r} \right]dr + \int_{t - h}^t \pars{ f(\tilde \gamma_r) - f(\gamma_r)} \cdot dB_r.
	\]
	To bound the first integral, we use \eqref{DHstar} and \eqref{comparativepaths} to obtain
	\begin{align*}
		\int_{t-h}^t \left[ H^*\pars{ \dot \gamma_r + \frac{\tilde y - y}{h}}  - H^*\pars{ \dot \gamma_r} \right]dr &\le
		C \int_{t-h}^t \pars{ 1 + |\dot \gamma_r|^{q'-1} + |\dot{\tilde \gamma}_r|^{q'-1} }dr \frac{|\tilde y - y|}{h}\\
		&\le C\pars{ |\tilde y -y|  + \frac{|\tilde y- y|}{h}\int_{t - h}^t \pars{  |\dot \gamma_r|^{q'-1} +  |\dot {\tilde \gamma}_r|^{q'-1} }dr }\\
		&\le C\pars{ |\tilde y -y| +  |\tilde y - y|^{\alpha q/(1+ \alpha q)} \pars{ 1 + \mcl D_{\sigma,\Lambda}}^{1/q} }. 
	\end{align*}
	Applying Lemma \ref{L:fauxIto} in conjunction with Lemma \ref{L:minimizer} gives
	\[
		\abs{ \int_{t - h}^t f(\gamma_r) \cdot dB_r} \le C \pars{ \int_s^t |\dot \gamma_r|^{q'}dr + 1 + \mcl D_{\sigma,\Lambda}}h^\alpha \le C(1 + \mcl D_{\sigma,\Lambda}) h^\alpha,
	\]
	and, because of \eqref{comparativepaths}, the same estimate holds with $\tilde \gamma$ in place of $\gamma$.
		
	Note that $h^\alpha \le C |\tilde y - y|^{\alpha q/(1+ \alpha q)}$. Then, by sending $\nu \to 0$ and exchanging the roles of $\tilde y$ and $y$, we conclude that, for some $M > 0$,
	\[
		\abs{ L_f(x,\tilde y,s,t) - L_f(x,y,s,t)} \le M(1 + \mcl D_{\sigma,\Lambda} ) |\tilde y- y|^{\alpha q/(1+ \alpha q)}.
	\]
	Similar arguments give, for all $(x,y,s,t), (\tilde x, y,s,t) \in U_{R,T,\tau}$,
	\[
		\abs{ L_f(\tilde x,y,s,t) - L_f(x,y,s,t)} \le M (1 + \mcl D_{\sigma,\Lambda} ) |\tilde x- x|^{\alpha q/(1+ \alpha q)}.
	\]
	
	{\it Step 2: time increments.} Fix $(x,y,s,t), (x,y,s,\tilde t) \in U_{R,T,\tau}$ with $\tilde t < t$. Then the sub-additivity of $L_f$ (see \eqref{Lsubadditive}) yields
	\[
		L_f(x,y,s, t) - L_f(x,y,s, \tilde t) \le L_f(y,y,\tilde t, t).
	\]
	Bounding $L_f(y,y,\tilde t, t)$ from above by applying the Lagrangian action to the constant path $\gamma \equiv y$ gives, for some $C = C(R,q) > 0$,
	\begin{equation}\label{Ltimeupper}
		L_f(x,y,s, t) - L_f(x,y,s, \tilde t) \le C(t - \tilde t) + f(y) \cdot (B(t) - B(\tilde t))\le C( 1 + \sigma [B]_{\alpha,[0,T]} ) (t - \tilde t)^\alpha.
	\end{equation}
	
	We next find a lower bound for $L_f(x,y,s, t) - L_f(x,y,s, \tilde t)$. For $\nu \in (0,1)$, let $\gamma \in \mcl A(x,y,s,t)$ once more satisfy \eqref{numin}, set
	\[
		h := c(t-s) ( t - \tilde t)^{1/(1+\alpha q)},
	\]
	where $c = c(\tau) >0$ is chosen so that $s < t - h < t$, and define $\tilde \gamma \in \mcl A(x,y,s, \tilde t)$ by
	\[
		\tilde \gamma_r :=
		\begin{dcases}
			\gamma_r & \text{for } r \in [s, \tilde t - h), \text{ and} \\
			\gamma_r + \frac{y - \gamma_{\tilde t}}{h}(r - \tilde t + h) & \text{for } r \in [\tilde t - h,\tilde t].
		\end{dcases}
	\]
	Then
	\begin{equation} \label{Ltimereg}
		\begin{split}
		L_f(x,y,s,\tilde t) - L_f(x,y,s, t) &\le \int_{ \tilde t -h}^{\tilde t} \left[ H^*\pars{ \dot \gamma_r + \frac{y -  \gamma_{\tilde t}}{ h}} - H^*\pars{ \dot \gamma_r} \right]dr - \int_{\tilde t}^t H^*(\dot \gamma_r)dr \\
		&+ \int_{\tilde t - h}^{\tilde t} \pars{ f(\tilde \gamma_r) - f(\gamma_r)} \cdot dB_r - \int_{\tilde t}^t f(\gamma_r) \cdot dB_r.
		\end{split}
	\end{equation}
	Observe that
	\[
		|y - \gamma_{\tilde t}| \le \int_{\tilde t}^t \abs{ \dot \gamma_r}dr \le \pars{ \int_s^t \abs{ \dot \gamma_r}^{q'} dr}^{1/q'} ( t - \tilde t)^{1/q},
	\]
	which implies that
	\begin{align*}
		\nor{\dot{\tilde \gamma}}{L^{q'}[s,\tilde t]} &\le \nor{\dot \gamma}{L^{q'}[s,t]} + \pars{ \int_{\tilde t - h}^{\tilde t} \abs{ \frac{y - \gamma_{\tilde t}}{h} }^{q'}dr}^{1/q'} \\
		&= \nor{\dot \gamma}{L^{q'}[s,t]} + \frac{|y - \gamma_{\tilde t} |}{h^{1/q}}\\
		&\le \nor{\dot \gamma}{L^{q'}[s,t]} + C \nor{\dot \gamma}{L^{q'}[s,t]} (t - \tilde t)^{\alpha/(1+\alpha q)}\\
		&\le C \nor{\dot \gamma}{L^{q'}[s,t]},
	\end{align*}
	and, therefore, in view of Lemma \ref{L:minimizer},
	\[
		\int_s^{\tilde t} \pars{ |\dot \gamma_r|^{q'} + |\dot{\tilde \gamma}_r|^{q'} }dr \le C(1 + \mcl D_{\sigma,\Lambda}).
	\]
	The first two integrals on the right-hand side of \eqref{Ltimereg} can then be bounded using \eqref{Hstarbounds} and \eqref{DHstar} to obtain
	\begin{align*}
	\int_{\tilde t -h}^{\tilde t} &\left[ H^*\pars{ \dot \gamma_r + \frac{y -  \gamma_{\tilde t}}{h}} - H^*\pars{ \dot \gamma_r} \right]dr - \int_{\tilde t}^t H^*(\dot \gamma_r)dr\\
		&\le C(t - \tilde t) + C\int_{\tilde t - h}^{\tilde t} \pars{ 1 + |\dot{\tilde \gamma}_r|^{q'-1} + | \dot \gamma_r|^{q'-1} }dr \frac{|y - \gamma_{\tilde t}|}{h}\\
		&\le C(t - \tilde t) + C h + C \int_{\tilde t - h}^{\tilde t} \pars{ |\dot{\tilde \gamma}_r|^{q'-1} + | \dot \gamma_r|^{q'-1} }dr\pars{ \int_s^t \abs{ \dot \gamma_r}^{q'} dr}^{1/q'} \frac{ ( t - \tilde t)^{1/q} }{h} \\
		&\le C(t - \tilde t) + C(t - \tilde t)^{\alpha q/(1+\alpha q)} + C(1 + \mcl D_{\sigma,\Lambda}) (t- \tilde t)^{\alpha/(1+\alpha q)} \\
		&\le C(1 + \mcl D_{\sigma,\Lambda}) (t- \tilde t)^{\alpha/(1+\alpha q)}.
	\end{align*}
	Additionally, Lemma \ref{L:fauxIto} gives
	\[
		\abs{ \int_{\tilde t - h}^{\tilde t} f(\gamma_r) \cdot dB_r }  + \abs{ \int_{\tilde t - h}^{\tilde t} f(\tilde \gamma_r) \cdot dB_r } \le C(1 + \mcl D_{\sigma,\Lambda}) h^{\alpha}
	\]
	and
	\[
		\abs{ \int_{\tilde t}^t f(\gamma_r) \cdot dB_r } \le C(1 + \mcl D_{\sigma,\Lambda})(t- \tilde t)^\alpha.
	\]
	Sending $\nu \to 0$, combining all of the lower bounds with the upper bound \eqref{Ltimeupper}, and applying a similar argument in the $s$ variable, we conclude that, for some $M > 0$,
	\begin{align*}
		\abs{ L(x,y,\tilde s, \tilde t) - L(x,y,s,t) } \le M(1 + \mcl D_{\sigma,\Lambda} + \sigma [B]_{\alpha,[0,T]} )\pars{ |s - \tilde s|^{\alpha/(1 + \alpha q)} + |t - \tilde t|^{\alpha/(1 + \alpha q) } }.
	\end{align*}
	
	{\it Step 3.} Combining all of the estimates obtained in Steps 1 and 2 gives
	\[
		\sup_{f \in  F_{\sigma,\Lambda}} [L_f]_{C_{x,y}^{\theta} C_{s,t}^{\theta/q}(U_{R,T,\tau}) } - M \le M(\mcl D_{\sigma,\Lambda} + \sigma [B]_{\alpha,[0,T]}).
	\]
	The result now follows as a consequence of Lemma \ref{L:fauxIto} and the fact that $[B]_{\alpha,[0,T]} \in L^p(\Omega,\mbf P)$ for all $p \ge 1$.
	
\end{proof}

Proposition \ref{P:Lregularity} can be used to obtain regularity estimates for solutions $u$ of
\begin{equation}\label{E:regularity}
	u_t + H(Du) = \sum_{i=1}^m f^i(x) \dot B^i(t,\omega) \quad \text{in } \RR^d \times (0,\oo) \times \Omega \quad \text{and} \quad u(\cdot,0,\cdot) = u_0 \quad \text{in } \RR^d \times \Omega.
\end{equation}
Although we do not directly use the following result in the later parts of the paper, its statement is of independent interest.

\begin{theorem}\label{T:regularity}
	Assume that $H$ satisfies \eqref{A:Hestimates}, $0 < \theta < \frac{q}{2+q}$, and $\mcl R > 1$. Then there exists a constant $C_1 = C_1(\mcl R,q,\theta) > 0$, and for $p \ge 1$, a constant $C_2 = C_2(\mcl R,p,q,\theta) > 0$ such that, whenever $B: [0,\oo) \times \Omega \to \RR^m$ is a standard Brownian motion; $f \in C^1_b(\RR^d;\RR^m)$ and $u_0 \in BUC(\RR^d)$ satisfy
	\[
		\nor{f}{\oo}\cdot \nor{Df}{\oo} + \nor{f}{\oo} + \nor{u_0}{\oo} \le \mcl R;
	\]
	and $u$ is the solution of \eqref{E:regularity}, then, for all $\lambda \ge 1$,
	\[
		\mbf P \pars{ \sup_{(x,s),(y,t) \in B_{\mcl R} \times [1/\mcl R,\mcl R]} \frac{ \abs{ u(x,s) - u(y,t)}}{|x-y|^\theta + |s-t|^{\theta/q}} > C_1 + \lambda } \le \frac{C_2 \nor{f}{\oo}^{p}}{\lambda^{p}}.
	\]
\end{theorem}

\begin{proof}
	The function $u$ is given by
	\[
		u(x,t) = \inf_{y \in \RR^d} \pars{ u_0(y) + L_f(y,x,0,t)}.
	\]
	Assume that $t > 1/\mcl R$ and $|x| \le \mcl R$. Then Lemma \ref{L:minimizer} gives the upper bound, for some $C = C(\mcl R,q,\theta) > 0$,
	\[
		u(x,t) \le u_0(x) + L_f(x,x,0,t) \le C(1 + \mcl D_{\sigma,\Lambda}).
	\]
	It then follows that the infimum in the definition of $u$ can be taken over $|y-x| \le R'$, where, for another constant $C = C(\mcl R,q,\theta) > 0$ and for all $\omega \in \Omega$,
	\[
		R'(\omega) := C(1 + \mcl D_{\sigma,\Lambda}(\omega)).
	\]
	
	It can be verified that the constant $M$ from Proposition \ref{P:Lregularity} has polynomial growth in $R$, that is, for some $M' = M'(T,\tau,\sigma_0, K,\theta,q) > 0$ and $a = a(\theta,q) > 0$,
	\[
		M(R,T,\tau,\sigma_0,K,\theta,q) \le M'(1 + R^a).
	\]
	We also have, for all $x,y \in B_{\mcl R}$ and $s,t \in [1/\mcl R,\mcl R]$,
	\[
		|u(x,s) - u(y,t)| \le \sup_{|x-z| \le R', |y-z| \le R'} \abs{L_f(z,x,0,s) - L_f(z,y,0,t)}.
	\]
	From this, we conclude that, for some $C_1 = C_1(\mcl R,q,\theta) > 0$ and $C_2 = C_2(\mcl R,p,q,\theta) > 0$ and for any $\lambda,\lambda' \ge 1$,
	\begin{align*}
		\mbf P \pars{ [u]_{C^\theta_x C^{\theta/q}_t(B_{1/\mcl R} \times [1/\mcl R,\mcl R] ) } > C_1 + \lambda}
		&\le \mbf P (R' > \lambda') + \mbf P \pars{ [L_f]_{C^\theta_{x,y} C^{\theta/q}_{s,t} (B_{\lambda'}^2 \times [1/\mcl R, \mcl R]^2 ) } > C_1 + \lambda}\\
		&\le C_2 \pars{ \frac{1}{(\lambda')^{p(1+a)}} + \frac{1 + (\lambda')^{pa(1+a)}}{\lambda^{p(1+a)}} } \nor{f}{\oo}^{p(1+a)}.
	\end{align*}
	The proof is finished upon choosing $\lambda' = \lambda^{1/(1+a)}$ and using the fact that $\nor{f}{\oo}^{1+a} \le \mcl R^a\nor{f}{\oo}$.
	
\end{proof}

\section{Homogenization}\label{S:homog}

We now turn to the proof of Theorem \ref{T:introhomog}, which concerns the convergence, as $\eps \to 0$, of solutions of the stochastically perturbed initial value problem
\begin{equation}\label{E:mainrepeat}
	u^\eps_t + H(Du^\eps) = F\pars{ \frac{x}{\eps}, \frac{t}{\eps},\omega}  \quad \text{in } \RR^d \times (0,\oo) \times \Omega \quad \text{and} \quad u^\eps(x,0,\omega) = u_0(x) \quad \text{in } \RR^d \times \Omega
\end{equation}
to the solution of the effective equation
\begin{equation}\label{E:limitrepeat}
	\oline{u}_t + \oline{H}(D \oline{u}) = 0 \quad \text{in } \RR^d \times (0,\oo) \quad \text{and} \quad \oline{u}(\cdot,0) = u_0 \quad \text{in } \RR^d.
\end{equation}

We restate the theorem here with more precise hypotheses and conclusions. Recall that $F$ is given by
\[
	F(x,t,\omega) = \sum_{i=1}^m f^i (x,\omega) \dot B^i(t,\omega) \quad \text{for } (x,t) \in \RR^d \times [0,\oo),
\]
where $f(\cdot,\omega) \in C^1_b(\RR^d)$ is a stationary-ergodic random field and $B$ is a standard $m$-dimensional Brownian motion that is independent of $f$.

Below, the constant $M_0 > 0$ is such that
\[
	\mbf P \pars{ \nor{f}{C^1} \le M_0 } = 1.
\]

\begin{theorem}\label{T:homog}
	Assume that $H$ satisfies \eqref{A:Hestimates} and the random field $F$ satisfies \eqref{A:randomfield}. Then there exists a deterministic, convex $\oline{H}: \RR^d \to \RR$ satisfying the bounds
	\begin{equation}\label{Hbarestimates}
		\frac{1}{C}|p|^q - C \le \oline{H}(p) \le C(|p|^q + 1) \quad \text{for some $C = C(M_0) > 1$ and all $p \in \RR^d$}
	\end{equation}
	such that, if $u_0 \in BUC(\RR^d)$ and $u^\eps$ solves \eqref{E:mainrepeat}, then, as $\eps \to 0$ and with probability one, $u^\eps$ converges locally uniformly to the viscosity solution $\oline{u}$ of \eqref{E:limitrepeat}.
\end{theorem}

Much of the proofs that follow proceed similarly to those in \cite{JSTvisc, LShomog, RT, Schwab, S}, with some new difficulties arising because of the singular nature of the forcing term.

The solution $u^\eps$ has the control-theory representation
\begin{equation}\label{uepsformula}
	u^\eps(x,t) := \inf_{y \in \RR^d} \left\{ u_0(y) + L^\eps(y,x,s,t,\omega) \right\},
\end{equation}
where
\begin{equation}\label{Leps}
	L^\eps(x,y,s,t,\omega) := \inf \left\{ \int_s^t \left[ H^*\pars{ \dot \gamma_r} + F\pars{ \frac{\gamma_r}{\eps},\frac{r}{\eps},\omega} \right] dr  : \gamma \in \mcl A(x,y,s,t) \right\}.
\end{equation}
Indeed, \eqref{uepsformula} holds if $B$ is replaced with a continuously differentiable path (see \cite{Lbook}), and, as a consequence of Lemmas \ref{L:forcingstability} and \ref{L:distancestability}, the equality continues to hold for arbitrary continuous $B$, and, in particular, sample paths of Brownian motion.

It will be useful to rewrite $L^\eps$ in two different ways. First, if we set $B^\eps(r,\omega) := \eps^{1/2} B(r/\eps,\omega)$, then
\begin{equation}\label{Lepsalt}
	L^\eps(x,y,s,t, \omega) = \inf \left\{ \int_s^t  H^*\pars{ \dot \gamma_r}dr + \eps^{1/2} \int_s^t f(\eps^{-1} \gamma_r,\omega) \cdot dB^\eps_r(\omega)   : \gamma \in \mcl A(x,y,s,t) \right\}.
\end{equation}
Also, by rescaling the paths $\gamma \in \mcl A(x,y,s,t)$, we see that
\[
	L^\eps(x,y,s,t,\omega) = \eps L\pars{ \frac{x}{\eps}, \frac{y}{\eps}, \frac{s}{\eps}, \frac{t}{\eps} ,\omega},
\]
where $L := L^1$ is given by
\begin{equation}\label{L}
	L(x,y,s,t,\omega) := \inf \left\{ \int_s^t  H^*\pars{ \dot \gamma_r}dr + \int_s^t f(\gamma_r,\omega) \cdot dB_r (\omega)  : \gamma \in \mcl A(x,y,s,t) \right\}.
\end{equation}
We note that, for $(x,y,s,t) \in \RR^d \times \RR^d \times [0,T] \times [0,T]$ and $\omega \in \Omega$,
\[
	L(x,y,s,t,\omega) = L_{f(\cdot,\omega)}(x,y,s,t,\omega),
\]
where, for fixed $f \in C^1_b(\RR^d,\RR^m)$, $L_f$ is defined as in \eqref{Lf}.

\subsection{Uniform regularity of the Langrangians}

We use Proposition \ref{P:Lregularity} and the Borel-Cantelli lemma to show that the family $(L^\eps)_{\eps > 0}$ is locally uniformly equicontinuous for $\eps$ smaller than some random threshold.

Recall that we define, for $R > 0$ and $0 < \tau < T$, the domain
\[
	U_{\tau,R,T} := \left\{ (x,y,s,t) \in \RR^d \times \RR^d \times [0,T] \times [0,T] : |x-y| \le R \text{ and } \tau < t - s < T \right\},
\]
and the parameters $\alpha \in (0,1/2)$ and $\theta \in (0,q/(2+q))$ are related by
\[
	\theta = \frac{\alpha q}{1 + \alpha q}.
\]

\begin{lemma}\label{L:Lepsreg}
	Let $L^\eps$ be given by \eqref{Leps}, where $H$ and $F$ are as in \eqref{A:Hestimates} and \eqref{A:randomfield}. Then, for all $0 < \theta < \frac{q}{2+q}$, $R > 0$, and $0 < \tau < T$, there exists a constant $C = C(R,T,\tau,M_0, q,\theta) > 0$ and a random variable $\eps_0: \Omega \to \RR_+$ independent of $f$ such that $\mbf P(\eps_0 > 0) = 1$ and
	\[
		\mbf P\pars{ \left[ L^\eps \right]_{C_{x,y}^{\theta} C_{s,t}^{\theta/q}(U_{R,T,\tau}) }\le C \; \text{for all } 0 < \eps < \eps_0 } = 1.
	\]
\end{lemma}

\begin{proof}
	For each fixed $\eps \in (0,1)$, we apply Proposition \ref{P:Lregularity} using the formula \eqref{Lepsalt} for $L^\eps$, with
	\[
		f^\eps := \eps^{1/2} f(\cdot/\eps) \quad \text{and} \quad B^\eps = \eps^{1/2} B(\cdot/\eps).
	\]
	Then $f^\eps \in F_{\sigma^\eps,\Lambda^\eps}$, where
	\[
		\sigma^\eps := M_0 \eps^{1/2}, \quad \Lambda^\eps := \frac{M_0}{\eps^{1/2}},
	\]
	and $F_{\sigma^\eps,\Lambda^\eps}$ is defined as in \eqref{C1subspace}. Note that, $K := \sigma^\eps \Lambda^\eps = M_0^2$ is independent of $\eps > 0$. Then there exists a constant $C_1 = C_1(R,T,\tau,M_0,q,\theta) > 0$ and, for all $p \ge 1$, a constant $C_2 = C_2(R,T,\tau,M_0,p,q,\theta) > 0$ such that, for all $\lambda \ge 1$ and $\eps \in (0,1)$,
	\[
		\mbf P \pars{ \sup_{f \in F_{\sigma^\eps,\Lambda^\eps} } [L_f]_{C_{x,y}^{\theta}C_{s,t}^{\theta/q}(U_{R,T,\tau})} > C_1 + \lambda} \le C_2 \eps^{p/2}\lambda^{-p}.
	\]
	Define
	\[
		A_{\eps} := \left\{  \sup_{f \in F_{\sigma^{\eps},\Lambda^{\eps}} } [L_f]_{C_{x,y}^{\theta}C_{s,t}^{\theta/q}(U_{R,T,\tau})} > C_1 + 1 \right\},
	\]
	and observe that $A_{\eps}$ is independent of the random field $\left\{ \omega \mapsto f(\cdot,\omega)\right\}$. Because
	\[
		\mbf P \pars{A_{2^{-k}}} \le C_2 2^{-kp/2},
	\]
	the Borel-Cantelli lemma yields that, for some $k_0: \Omega \to \NN$ that is independent of the random field $f$, for $\mbf P$-almost every $\omega \in \Omega$, and for all $k \ge k_0(\omega)$,
	\[
		\left[ L^{2^{-k}}(\cdot,\omega) \right]_{C_{x,y}^{\theta}C_{s,t}^{\theta/q}(U_{R,T,\tau})} \le C_1 + 1.
	\]
	For $\omega \in \Omega$ and $0 < \eps < \eps_0(\omega) := 2^{-k_0(\omega)}$, we choose $k > k_0(\omega)$ such that
	\[
		2^{-k-1} < \eps \le 2^{-k}.
	\]
	Then
	\[
		L^\eps(x,y,s,t,\cdot) = \eps L\pars{ \frac{x}{\eps}, \frac{y}{\eps}, \frac{s}{\eps}, \frac{t}{\eps} ,\cdot} = 2^{k+1}\eps L^{2^{-k-1}} \pars{2^{-k-1}\eps^{-1} x, 2^{-k-1}\eps^{-1}y, 2^{-k-1}\eps^{-1}s, 2^{-k-1}\eps^{-1} t,\cdot},
	\]
	and therefore, for $\mbf P$-almost every $\omega \in \Omega$ and all $\eps \in (0,\eps_0(\omega))$,
	\[
		[L^{\eps}(\cdot,\omega)]_{C_{x,y}^{\theta}C_{s,t}^{\theta/q}(U_{R,T,\tau})} \le 2 (C_1+1).
	\]
\end{proof}

\subsection{The stationary-ergodic, spatio-temporal environment}\label{SS:productmeasure}

The temporal white noise term $\dot B$ is stationary, uncorrelated, and independent from $f$, and, as a consequence, the spatio-temporal environment generated by the random field $F$ is stationary-ergodic. More precisely, we may assume without loss of generality that the probability measure $\mbf P$ is such that there exists a collection of transformations
\[
	(\tau'_t)_{t \ge 0}: \Omega \to \Omega
\]
such that, for all $s, t \ge 0$ and $\omega \in \Omega$,
\[
	\tau'_{s+t} = \tau'_s \circ \tau'_t,
	\quad
	\mbf P \circ \tau'_t = \mbf P, \quad 
	B(s,\tau'_t \omega) = B(t+s,\omega) - B(t,\omega), 
	\quad \text{and} \quad
	f(\cdot,\tau'_t \omega) = f(\cdot,\omega).
\]
For $(x,t) \in \RR^d \times [0,\oo)$, we now set
\[
	\bs{\tau}_{x,t} := \tau_x \circ \tau'_t : \Omega \to \Omega.
\]
It is clear that $\bs{\tau}_{x,t}$ preserves $\mbf P$ for any $(x,t) \in \RR^d \times [0,\oo)$. Moreover the collection is ergodic with respect to $\mbf P$:
\[
	\text{if } A \in \mbf F \text{ and } \bs\tau_{x,t} A = A \text{ for all } (x,t) \in \RR^d \times [0,\oo), \text{ then } \mbf P\pars{ A }\in \{0, 1\}.
\]
As noted in Section \ref{S:A}, we may assume that $\Omega = X \times Y$ and $\mbf P = \mu \otimes \nu$, where $X = C^1_b(\RR^d,\RR^m)$, $Y \in C([0,\oo),\RR^m)$, $\mu$ is a probability measure that is stationary and ergodic with respect to translations in space, and $\nu$ is the Wiener measure. The transformation $\tau'_t$ can then be realized as
\[
	\tau_t' B := B(t + \cdot) - B(t) \quad \text{for all } B \in Y,
\]
and the stationarity of $\nu$ with respect to $\tau_t'$ is a consequence of the Markov property of Brownian motion.

\subsection{Identification of the effective Lagrangian}
We next use the sub-additive ergodic theorem to establish the almost-sure, local uniform convergence of $L^\eps$ to a deterministic, effective quantity. This relies on the sub-additivity and stationarity of $L$ defined by \eqref{L}, namely, for all $x,y,z \in \RR^d$, $s < r < t$, $q \in [0,\oo)$, and $\omega \in \Omega$,
\begin{equation}\label{Lstationary}
	L(x,y,s,t,\bs \tau_{z,q} \omega) = L(x + z,y+z,s+q,t+q,\omega)
\end{equation}
and
\begin{equation}\label{Lsubadditive}
	L(x,y,s,t,\omega) \le L(x,z,s,r,\omega) + L(z,y,r,t,\omega),
\end{equation}
both of which can proved using appropriate manipulations of the minimizing paths in the definition of $L$, invoking the stationarity of $F$ to prove the former.

We first identify the effective Lagrangian as the long-time average of $L$.

\begin{lemma}\label{L:identifyLbar}
	Assume \eqref{A:Hestimates}, \eqref{A:randomfield}, and that $L$ is given by \eqref{L}. Then there exists a deterministic, convex function $\oline{L}: \RR^d \to \RR$ such that, for some $C = C(M_0) > 1$,
	\begin{equation}\label{Lbarbounds}
		\frac{1}{C} |p|^{q'} - C \le \oline{L}(p) \le C (|p|^{q'} + 1) \quad \text{for all } p \in \RR^d,
	\end{equation}
	and, with probability one and for all $R > 0$,
	\[
		\lim_{T \to +\oo}  \sup_{|p| \le R} \abs{\frac{1}{T} L(0,Tp,0,T,\cdot) - \oline{L}(p)} = 0.
	\]	
\end{lemma}

\begin{proof}
	{\it Step 1: identifying the limit.} Fix $p \in \QQ^d$, define the process $\phi$ by
	\[
		\phi\pars{ [a,b), \omega} := L(ap, bp, a,b,\omega) \quad \text{for } 0 \le a < b,
	\]
	and, for $t \ge 0$, define the measure-preserving transformation $\sigma_t: \Omega \to \Omega$ by $\sigma_t := \bs\tau_{tp,t}$.
	
	In order to apply the sub-additive ergodic theorem of Akcoglu and Krengel \cite{AK}, we need to verify that $\phi$ is a stationary, sub-additive process with respect to $(\sigma_t)_{t \ge 0}$, that is, for all $\omega \in \Omega$,
	\begin{equation}\label{applyAK}
		\left\{
		\begin{split}
			&\phi \pars{ [a,b),\sigma_t \omega} = \phi \pars{ [a,b) + t,\omega} \quad \text{for all $a < b$ and $t \ge 0$,}\\
			&\phi \pars{ [a,c),\omega} \le \phi \pars{ [a,b),\omega} + \phi\pars{ [b,c),\omega} \text{ for all $a < b < c$, and}\\
			&\inf_{T > 0} \frac{1}{T} \mbf E \phi \pars{ [0,T),\cdot} > -\oo.
		\end{split}
		\right.
	\end{equation}
	
	The stationarity and sub-additivity of $\phi$ follow from \eqref{Lstationary} and \eqref{Lsubadditive}, which give
	\[
		\phi\pars{[a,b),\sigma_t \omega} = L(ap,bp,a,b,\bs\tau_{tp,t}\omega) = L((a+t)p,(b+t)p,a+t,b+t,\omega) = \phi\pars{[a,b) + t,\omega}
	\]
	and
	\[
		\phi\pars{ [a,c),\omega} = L(ap,cp,a,c,\omega) \le L(ap,bp,a,b,\omega) + L(bp,cp,b,c,\omega) = \phi\pars{ [a,b),\omega} + \phi\pars{ [b,c),\omega}.
	\]
	It remains to prove the third item of \eqref{applyAK}. This requires bounds for the long-time averages 
	\[
		\frac{1}{T} \int_0^T f(\gamma_r) \cdot dB_r.
	\]
	The difficulty is that, as $T \to \oo$, this quantity need not converge to $0$, in view of the fact that the almost-minimizing path $\gamma$ need not be adapted to the Brownian motion $B$, as was the case in the proof of Lemma \ref{L:fauxIto}.
		
	Fix $\gamma \in \mcl A(0,Tp,0,T)$ and $h > 0$ to be determined, and let $N \in \NN$ be such that $T/h \le N < T/h + 1$. Set $\tau_k := kh$ for $k = 0,1,2,\ldots, N-1$, $\tau_N =T$, and
	\[
		X_k := \max_{r \in [\tau_{k-1},\tau_k] } \abs{ B_r - B_{\tau_{k-1}} }.
	\]
	Then Young's inequality gives, for any $\delta > 0$ and some $C = C(q) > 0$,
	\begin{align*}
		\frac{1}{T} \int_0^T f(\gamma_r) \cdot dB_r
		&= \frac{1}{T} \sum_{k=1}^N \int_{\tau_{k-1}}^{\tau_k} f(\gamma_r) \cdot dB_r \\
		&= \frac{1}{T} \sum_{k=1}^N \pars{ f(\gamma_{\tau_k}) \cdot ( B_{\tau_k} - B_{\tau_{k-1}} ) - \int_{\tau_{k-1}}^{\tau_k} Df(\gamma_r) \cdot \dot \gamma_r \cdot (B_r - B_{\tau_{(k-1)h}})dr }\\
		&\ge - \frac{\delta^{q'}}{T} \int_0^T |\dot \gamma_r|^{q'}dr - \frac{\nor{f}{\oo}}{T} \sum_{k=1}^N X_k - \frac{Ch \nor{Df}{\oo}^q}{\delta^{q}T} \sum_{k=1}^N X_k^q.
	\end{align*}
	Choosing the constant $\delta$ small enough relative to the lower bound for $H^*$ in \eqref{Hstarbounds} gives, for some $C > 1$,
	\begin{align*}
		\frac{1}{T} \int_0^T H^*(\dot \gamma_r)dr + \frac{1}{T} \int_0^T f(\gamma_r) \cdot dB_r
		&\ge \frac{1}{CT} \int_0^T |\dot \gamma_r|^{q'}dr - C\pars{ 1 +  \frac{\nor{f}{\oo}}{T} \sum_{k=1}^N X_k + \frac{h \nor{Df}{\oo}^q}{T} \sum_{k=1}^N X_k^q}\\
		&\ge - C\pars{ 1 +  \frac{\nor{f}{\oo}}{T} \sum_{k=1}^N X_k + \frac{h \nor{Df}{\oo}^q}{T} \sum_{k=1}^N X_k^q}.
	\end{align*}
	Taking expectations, we find that, for some constant $C > 0$ independent of $T$,
	\begin{align*}
		\frac{1}{T} \mbf E L(0,Tp,0,T,\omega) &\ge -C\pars{ 1 +  \frac{\nor{f}{\oo} Nh^{1/2}}{T}+ \frac{\nor{Df}{\oo}^q N h^{1 + q/2}}{T}} \\
		&\ge -C\pars{ 1 +  \frac{\nor{f}{\oo} }{h^{1/2}}+  \nor{Df}{\oo}^q h^{q/2} }.
	\end{align*}
	We now choose 
	\[
		h := \frac{ \nor{f}{\oo}^{2/(1+q)}}{ \nor{Df}{\oo}^{2q/(1+q)} },
	\]
	which leads to the lower bound
	\[
		\frac{1}{T} \mbf E L(0,Tp,0,T,\omega) \ge -C(1 + K^{q/(q+1)} ),
	\]
	and so \eqref{applyAK} is proved.
	
	The sub-additive ergodic theorem of \cite{AK} then yields the existence of $\Omega_0 \in \mbf F$ with $\mbf P(\Omega_0) = 1$ and a random field $\oline{L}: \QQ^d \times \Omega \to \RR$ such that
	\[
		\oline{L}(p,\omega) =\lim_{T \to +\oo} \frac{1}{T} \phi([0,T),\omega) = \lim_{T \to +\oo} \frac{1}{T} L(0,Tp,0,T,\omega) \quad \text{for all } \omega \in \Omega_0  \text{ and }  p \in \QQ^d.
	\]	
	
	For $T > 0$, $p \in \RR^d$, and $\omega \in \Omega_0$, set
	\[
		\ell_T(p,\omega) := \frac{1}{T} L(0,Tp,0,T,\omega) = L^{1/T}(0,p,0,1,\omega).
	\]
	Fix $N \in \NN$. Then, by Lemma \ref{L:Lepsreg}, there exists $\Omega^N \in \mbf F$ such that $\Omega^N \subset \Omega_0$ and $\mbf P(\Omega^N) = 1$, and, for some $\eps_N : \Omega^N \to \RR_+$ and for all $\omega \in \Omega^N$, the collection
	\[
		\pars{ \ell_T(\cdot,\omega) }_{T > \eps_N^{-1}(\omega)}
	\]
	is equicontinuous on $B_N$. Set
	\[
		\Omega_1 := \bigcap_{N \in \NN} \Omega^N.
	\]
	Then $\mbf P(\Omega_1) = 1$, and, for any $\omega \in \Omega_1$, $\oline{L}(\cdot,\omega)$ can be extended to a continuous function on all of $\RR^d$, and moreover, as $T \to \oo$, $\ell_T(\cdot,\omega)$ converges locally uniformly to $\oline{L}(\cdot,\omega)$. 
	
	{\it Step 2: the limit is deterministic.} Fix $p \in \RR^d$ and $(y,s) \in \RR^d \times [0,\oo)$. Then Lemma \ref{L:Lepsreg} implies that there exists a modulus $\rho:[0,\oo) \to [0,\oo)$ such that, for all $\omega \in \Omega_1$ and sufficiently large $T$,
	\begin{align*}
		\abs{ \oline{L}(p,\omega) - \frac{1}{T} L(0,Tp, 0, T, \bs\tau_{y,s} \omega) } &= \abs{ \oline{L}(p,\omega) -L^{1/T}\pars{ \frac{y}{T} , p + \frac{y}{T}, \frac{s}{T}, 1 + \frac{s}{T},\omega} }\\
		&\le \abs{ \oline{L}(p,\omega) -L^{1/T}(0,p, 0, 1,\omega) } + \rho\pars{ \frac{|y|}{T} + \frac{s}{T}}.
	\end{align*}
	Sending $T \to \oo$ yields
	\[
		\oline{L}(p,\omega) = \oline{L}(p,\bs\tau_{y,s}\omega)
	\]
	for all $p \in \RR^d$, $(y,s) \in \RR^d \times [0,\oo)$, and $\omega \in \Omega_1$. The ergodicity of the group $(\bs \tau_\cdot)$ then implies that $\omega \mapsto \oline{L}(\cdot,\omega)$ is constant, and, in fact,
	\[
		\oline{L}(p) := \lim_{T \to \oo} \frac{1}{T} \mbf E L(0,Tp,0,T,\cdot).
	\]
	
	{\it Step 3: convexity and estimates}. The convexity can be seen in a standard way from the sub-additivity of $L$ (see, for instance, \cite{Schwab,S}). To prove \eqref{Lbarbounds}, we appeal to Lemma \ref{L:minimizer}, which yields a constant $C = C(M_0,q) > 1$ such that, for all $T > 0$,
\[
	\frac{1}{C} |p|^{q'} - C \le \mbf E L^{1/T}(0,p,0,1,\omega) \le C \pars{ |p|^{q'} + 1}.
\]
Letting $T \to \oo$ finishes the proof.
\end{proof}

Before we continue, we give an example to show that, in general, the assumption of almost-sure boundedness for $f$ cannot be dropped in the above argument. Consider the stationary sub-additive process
\[
	\psi([a,b),\omega) := \inf\left\{ \frac{1}{2} \int_a^b |\dot \gamma_t|^2 dt + \int_a^b \gamma_t\cdot dB_t : \gamma \in \mcl A(0,0,a,b) \right\};
\]
here $d = m = 1$ and $B$ is a one-dimensional Brownian motion. A straightforward computation reveals that the action is minimized by the path
\[
	[a,b] \ni t \mapsto \gamma^*_t := \int_a^t \pars{ B_s - \frac{1}{b-a} \int_a^b B_r dr  }ds,
\]
which leads to the identity
\[
	\psi([a,b),\omega) = \frac{1}{2} \left[ \frac{1}{b-a} \pars{ \int_a^b B_r dr}^2 - \int_a^b B_r^2 dr \right].
\]
Then
\[
	\mbf E \frac{1}{T} \psi([0,T),\cdot) = - \frac{T}{12},
\]
which is unbounded in $T$.

\subsection{The local uniform convergence of $L^\eps$}

Fix $\eps > 0$ and $(x,y,s,t) \in \RR^d \times \RR^d \times [0,\oo) \times [0,\oo)$ with $s < t$. Then Lemma \ref{L:identifyLbar} and the stationarity of $L^\eps$ yield
\[
	\mbf E L^\eps(x,y,s,t,\cdot) = \mbf E \eps L \pars{0, \frac{y-x}{\eps}, 0, \frac{t-s}{\eps},\cdot } \xrightarrow{\eps \to 0} (t-s) \oline{L}\pars{ \frac{y-x}{t-s} }.
\]
The next lemma establishes the local-uniform convergence of $L^\eps$ with probability one, using a standard argument that combines the multi-parameter ergodic theorem and Egoroff's theorem. Such an argument has been used several times throughout the literature on the stochastic homogenization of Hamilton-Jacobi equations; see, for example, \cite{ASunbounded,AS,AT,CST,JSTfront,JSTvisc,KRV,KV,Schwab}.

\begin{lemma}\label{L:Lbarlocaluniform}
	Assume \eqref{A:Hestimates} and \eqref{A:randomfield}. Then
	\[
		\mbf P \pars{ \lim_{\eps \to 0} \max_{|x| \le R} \max_{ |y| \le R} \max_{s,t \in [0,T], t-s > \tau} \abs{ L^\eps(x,y,s,t,\cdot) - (t-s) \oline{L}\pars{ \frac{y-x}{t-s}} }  = 0 \text{ for all $R > 0$, $0 < \tau < T$} } = 1.
	\]
\end{lemma}

\begin{proof}
	Let $\Omega_0 \in \mbf F$ be the event of full probability for which the conclusion of Lemma \ref{L:identifyLbar} holds. Then Egoroff's Theorem implies that, for all $\eta >0$, there exists an event $G_\eta \subset \Omega_0$ such that $\mbf P(G_\eta) \ge 1 - \eta$ and, for any $M > 0$,
	\[
		\lim_{T \to \oo} \sup_{|p| \le M} \sup_{\omega \in G_\eta} \abs{ \frac{1}{T} L\pars{ 0, Tp, 0, T,\omega} - \oline{L}(p)} = 0.
	\]
	
	Fix $R > 0$ and $T > 0$, and, for $\omega \in \Omega_0$, define
	\[
		A^\eps_\eta(\omega) := \left\{ (x,s) \in B_R \times [0,T] : \bs\tau_{x/\eps,s/\eps} \omega \in G_\eta \right\}.
	\]
	The multiparameter ergodic theorem (see Becker \cite{Be}) then yields that, for some $\Omega_\eta \in \mbf F$ satisfying $\Omega_\eta \subset \Omega_0$ and $\mbf P(\Omega_\eta) = 1$, and for all $\omega \in \Omega_\eta$,
	\[
		\frac{ \abs{ A^\eps_\eta(\omega)}}{\abs{ B_R \times [0, T]}} = \frac{1}{ \abs{ B_R \times [0,T]}} \int_{B_R \times [0,T]} \ind_{G_\eta}\pars{ \bs\tau_{(x/\eps,s/\eps)} \omega} dxds \xrightarrow{\eps \to 0} \mbf P(G_\eta),
	\]
	and so, for some $\eps_\eta: \Omega \to \RR_+$ and for all $\omega \in \Omega_\eta$, $\eps_\eta(\omega) > 0$ and
	\[
		\abs{ A^\eps_\eta(\omega)} \ge (1 - 2 \eta) \abs{ B_R \times [0, T]} \quad \text{for all } 0 < \eps < \eps_1(\omega).
	\]
	
	Now, fix $\omega \in \Omega_\eta$, $0 < \eps < \eps_\eta(\omega) \wedge \eps_0(\omega)$, $(x,y) \in B_R^2$, and $s,t \in [0,T]$ satisfying $t-s \ge \tau$. Then there exists $(x_\eps,s_\eps) \in A^\eps_\eta(\omega)$ such that, for some constant $c = c(d,R,T)> 0$,
	\[
		|x - x_\eps| + |s - s_\eps| \le c \eta^{1/(d+1)}.
	\]
	Choosing $\eta$ sufficiently small, depending on $\tau$, we have $t- s_\eps > \tau/2$.
	
	Set
	\[
		p_\eps := \frac{y-x_\eps}{t-s_\eps} \quad \text{and} \quad T_\eps := \frac{t-s_\eps}{\eps}.
	\]
	Note that, for some constant $M = M(R,\tau) > 0$, $\abs{ p_\eps} \le M$.
	
	We now invoke Lemma \ref{L:Lepsreg}, which gives a deterministic modulus $\rho: [0,\oo) \to [0,\oo)$, depending only on $R$, $T$, and $\tau$, such that, for all $\omega \in \Omega_\eta$ and $\eps \in (0, \eps_0(\omega) \wedge \eps_\eta(\omega))$,
	\begin{align*}
		&\abs{ L^\eps(x,y,s,t,\omega) - (t-s) \oline{L}\pars{ \frac{y-x}{t-s}} } 
		\le \abs{ L^\eps(x_\eps,y,s_\eps,t,\omega) - (t - s_\eps) \oline{L}\pars{ \frac{y - x_\eps}{t - s_\eps}}}\\
		&+ \abs{ L^\eps(x,y,s,t,\omega) - L^\eps(x_\eps,y,s_\eps,t,\omega)} + \abs{ (t - s) \oline{L}\pars{ \frac{y-x}{t-s}} - (t - s_\eps) \oline{L}\pars{ \frac{y-x_\eps}{t-s_\eps}} }\\
		&\le \rho(\eta) + T\abs{ \frac{1}{T_\eps} L\pars{ 0, T_\eps p_\eps, 0, T_\eps, \bs\tau_{(x_\eps/\eps, s_\eps/\eps)} \omega} - \oline{L}(p_\eps)} \\
		&\le \rho(\eta) + T \sup_{|p| \le M} \sup_{\tilde \omega \in G_\eta} \abs{ \frac{1}{T_\eps} L(0, T_\eps p, 0, T_\eps,\tilde \omega) - \oline{L}(p)}.
	\end{align*}
	Sending $\eps \to 0$ gives
	\[
		\limsup_{\eps \to 0} \max_{|x| \le R} \max_{ |y| \le R} \max_{s,t \in [0,T], t-s > \tau} \abs{ L^\eps(x,y,s,t,\omega) - (t-s) \oline{L}\pars{ \frac{y-x}{t-s}} } \le \rho(\eta).
	\]
	For $n \in \NN$, set $\eta_n := \frac{1}{n}$. If
	\[
		\tilde \Omega := \bigcup_{n \in \NN} \Omega_{\eta_n},
	\]
	we then have $\mbf P(\tilde \Omega) = 1$ and, for all $\omega \in \tilde \Omega$,
	\[
		\lim_{\eps \to 0} \max_{|x| \le R} \max_{ |y| \le R} \max_{s,t \in [0,T], t-s > \tau} \abs{ L^\eps(x,y,s,t,\omega) - (t-s) \oline{L}\pars{ \frac{y-x}{t-s}} } = 0.
	\]
	The result is finished in a standard way upon repeating the argument for a countable collection of $R$, $T$, and $\tau$, and taking the countable intersection of the resulting events of full probability.
\end{proof}

\subsection{The effective Hamiltonian and the homogenization of the equation}

Define the convex function
\[
	\oline{H}(p) := \oline{L}^*(p) := \sup_{q \in \RR^d} \pars{ p\cdot q - \oline{L}(q)},
\]
which, in view of \eqref{Lbarbounds}, immediately satisfies \eqref{Hbarestimates} (although in the following section, we will provide a sharper lower bound).

The proof of Theorem \ref{T:homog} will follow from the local uniform convergence of the Lagrangians, as well as the fact that, because $\oline{H}$ is convex and $u_0 \in BUC(\RR^d)$, $\oline{u}$ is given by the Hopf-Lax formula
\begin{equation}\label{E:HopfLax}
	\oline{u}(x,t) = \inf_{y \in \RR^d} \pars{ u_0(y) + t\oline{L}\pars{ \frac{x-y}{t}} }.
\end{equation}

\begin{proof}[Proof of Theorem \ref{T:homog}]

Let $\Omega_1 \in \mbf F$ be the event of full probability for which the conclusion of Lemma \ref{L:Lbarlocaluniform} holds.

{\it Step 1: coercivity bounds.} We first demonstrate that there exists $\Omega_2 \subset \Omega_1$ such that $\mbf P (\Omega_2) = 1$, as well as $\eps_2: \Omega \to \RR_+$ and a constant $C = C(T,M_0,q,\alpha) > 1$ such that, for all $x,y \in \RR^d$, $0 \le s < t \le T$, $\omega \in \Omega_2$, and $0 < \eps < \eps_2(\omega)$,
\begin{equation}\label{Lepscoercive}
	-C(t-s)^\alpha + \frac{1}{C} \frac{|y-x|^{q'}}{(t-s)^{q'-1}} \le L^\eps(y,x,s,t,\omega) \le C \frac{|y-x|^{q'}}{(t-s)^{q'-1} } + C(t-s)^\alpha.
\end{equation}
In view of Lemma \ref{L:minimizer}, there exist random variables $(\mcl D_\eps)_{\eps > 0}: \Omega \to \RR_+$ such that, for some constant $C = C(T,M_0,q,\alpha) > 1$ and for all $\eps > 0$, $x,y \in \RR^d$, and $0 \le s < t \le T$,
\[
	-C(\mcl D_\eps + 1) (t-s)^\alpha + \frac{1}{C} \frac{|y-x|^{q'}}{(t-s)^{q'-1}} \le L^\eps(y,x,s,t,\omega) \le C \frac{|y-x|^{q'}}{(t-s)^{q'-1} } + C(\mcl D_\eps + 1)(t-s)^\alpha,
\]
and, for all $p \ge 1$, there exists a constant $C' = C'(T,M_0,p,q,\alpha) > 0$ such that, for all $\eps > 0$ and $\lambda \ge 1$,
\[
	\mbf P\pars{ \mcl D_\eps > \lambda} \le \frac{C' \eps^{p/2}}{\lambda^{p}}.
\]
The Borel-Cantelli lemma yields the existence of $k_0: \Omega \to \NN$ and $\Omega_2 \subset \Omega_1$ with $\mbf P(\Omega_2) = 1$ such that, for all $\omega \in \Omega_2$ and $k \ge k_0(\omega)$,
\[
	\mcl D_{2^{-k}}(\omega) \le 1, 
\]
and so, for all $\omega \in \Omega_2$, $k \ge k_0(\omega)$, $x,y \in \RR^d$, and $s,t \in [0,T]$ with $s < t$,
\[
	-C(t-s)^\alpha + \frac{1}{C} \frac{|y-x|^{q'}}{(t-s)^{q'-1}} \le L^{2^{-k}}(y,x,s,t,\omega) \le C \frac{|y-x|^{q'}}{(t-s)^{q'-1} } + C(t-s)^\alpha.
\]
Now choose $\eps < \eps_2(\omega) := 2^{-k_0(\omega)}$ and let $k > k_0(\omega)$ be such that $2^{-k-1} < \eps \le 2^{-k}$. Then, for all $x,y \in \RR^d$ and $0 \le s < t \le T$,
\[
	L^\eps(x,y,s,t,\omega) = \eps L\pars{ \frac{x}{\eps}, \frac{y}{\eps}, \frac{s}{\eps}, \frac{t}{\eps},\omega } = (2^k \eps) L^{2^{-k}} \pars{ \frac{x}{2^k \eps}, \frac{y}{2^k\eps}, \frac{s}{2^k\eps}, \frac{t}{2^k\eps},\omega}.
\]
It follows that \eqref{Lepscoercive} holds upon replacing $C$ with $2C$, because $1 \le 2^k\eps < 2$.

{\it Step 2: localization.} Let $\tau > 0$ be fixed. We claim that there exists a deterministic $M$ depending only on $T$, $\tau$, and $\nor{u_0}{\oo}$ such that, for all $(x,t) \in \RR^d \times [0,T]$, $\omega \in \Omega_2$, and $\eps \in (0,\eps_2(\omega))$,
\begin{equation}\label{uepslocal}
	u^\eps(x,t,\omega) = \inf_{y \in B_M(x)} \pars{ u_0(y) + L^\eps(y,x,0,t,\omega)}.
\end{equation}
Setting $y = x$ in the definition of $u^\eps$ and using \eqref{Lepscoercive}, we find that there exists a constant $C = C(T) > 0$ such that, for all $\omega \in \Omega_2$, $\eps \in (0,\eps_2(\omega))$, and $(x,t) \in \RR^d \times [0,T]$,
\[
	u^\eps(x,t,\omega) \le C.
\]
The lower bound in \eqref{Lepscoercive} then yields that, if $|y-x| > M$ and $M$ is chosen large enough depending only on $T$ and $\tau$, then, for all $\eps \in (0,\eps_2(\omega))$,
\[
	u_0(y) + L^\eps(y,x,0,t,\omega) \ge -\nor{u_0}{\oo} + \frac{1}{C} \frac{|y-x|^{q'}}{\tau^{q'}} - C \ge - \nor{u_0}{\oo} + \frac{M^{q'}}{C\tau} - C > u^\eps(x,t,\omega).
\]
This establishes \eqref{uepslocal}.

{\it Step 3.} By \eqref{uepslocal} and Lemma \ref{L:Lbarlocaluniform}, we have, for all $\omega \in \Omega_2$,
\[
	\sup_{(x,t) \in B_R \times [\tau,T]} \abs{u^\eps(x,t,\omega) - \oline{u}(x,t)} \le \sup_{x \in B_R} \sup_{y \in B_{R+M}} \sup_{\tau < t < T} \abs{ L^\eps(y,x,0,t,\omega) - t \oline{L} \pars{ \frac{x-y}{t}} } \xrightarrow{\eps \to 0} 0.
\]
The bounds in \eqref{Lepscoercive} give, for $(x,t) \in B_R \times [0,\tau]$ and $\omega \in \Omega_2$, 
\[
	u^\eps(x,t,\omega) \le u_0(x) + C \tau^{\alpha},
\]
and, if $\rho: [0,\oo) \to [0,\oo)$ is the modulus of continuity for $u_0$,
\[
	u^\eps(x,t,\omega) \ge u_0(x) - C \tau^{\alpha} + \inf_{r \ge 0} \pars{ -\rho(r) + \frac{1}{C} \frac{r^{q'}}{\tau^{q'-1}} }.
\]
Define
\[
	\tilde \rho(\tau) :=  \sup_{r \ge 0} \pars{ \rho(r) - \frac{1}{C} \frac{r^{q'}}{\tau^{q'-1}} },
\]
and note that $\tilde \rho: [0,\oo) \to [0,\oo)$ satisfies $\lim_{\tau \to 0^+} \tilde \rho(\tau) = 0$. As a result, for all $\omega \in \Omega_2$ and $\eps \in (0,\eps_2(\omega))$,
\[
	\sup_{(x,t) \in B_R \times [0,\tau]} |u^\eps(x,t,\omega) - u_0(x)| \le C \tau^{\alpha} + C \tilde \rho(\tau).
\]
A similar argument gives 
\[
	\sup_{(x,t) \in B_R \times [0,\tau]} |\oline{u}(x,t,\omega) - u_0(x)| \le C \tau + \tilde \rho(\tau),
\]
and so, for all $\tau > 0$ and $\omega \in \Omega_2$,
\[
	\limsup_{\eps \to 0} \sup_{(x,t) \in B_R \times [0,T]} \abs{ u^\eps(x,t,\omega) - \oline{u}(x,t) } \le C(\tau + \tau^{\alpha} )+ \tilde \rho(\tau).
\]
The proof is finished upon letting $\tau \to 0$.
\end{proof}

\section{Enhancement}\label{S:enhancement}

If $L$ is defined as in \eqref{L}, then, for all $T > 0$ and $v \in \RR^d$,
\[
	\mbf E\frac{1}{T} L(0,Tv,0,T,\cdot) \le H^*(v) + \frac{1}{T}\mbf E \int_0^T f (tv,\cdot)\cdot dB_t = H^*(v),
\]
so that, in general, $\oline{L} \le H^*$ and
\[
	\oline{H} \ge H.
\]
This is actually an equality if $f$ is equal to a fixed, deterministic constant $\oline{f} \in \RR^d$, since then, for each $v \in \RR^d$ and for $\mbf P$-almost every $\omega \in \Omega$,
\[
	\frac{1}{T} L(0,Tv,0,Tv,\omega) = H^*(v) +  \oline{f} \cdot \frac{B(T,\omega)}{T} \xrightarrow{T \to \oo} H^*(v).
\]
It turns out that this is the only situation in which the two Hamiltonians are equal. In fact, a consequence of the isotropic nature of the temporal noise is that, if $f$ is nonconstant, then $\oline{H}$ is actually greater than $H$ everywhere. 

In order to facilitate the following arguments, we will assume additionally that
\begin{equation}\label{A:fC1kappa}
	\left\{
	\begin{split}
		&\text{for some $\kappa \in (0,1)$ and $M_0 > 0$,}\\
		&\text{$f(\cdot,\omega) \in C^{1,\kappa}(\RR^d,\RR^m)$ with probability one, and}\\
		&\nor{f}{\oo} + \nor{Df}{\oo} + [Df]_{\kappa} \le M_0.
	\end{split}
	\right.
\end{equation}
Let $\mbf v: \RR^d \to \RR^d$ and $\mbf p: \RR^d \to \RR^d$ satisfy
\[
	\mbf p(v) \in \argmax_{p \in \RR^d} \left\{ p \cdot v - H(p) \right\} \quad \text{and} \quad \mbf v(p) \in \argmax_{v \in \RR^d} \left\{ p \cdot v - H^*(v) \right\},
\]
and, for $\rho > 0$ and $v \in \RR^d$, set
\[
	G(v) := \sup_{\rho \in (0,1] }\frac{1}{\rho} \max_{|z| \le \rho} \left\{ H^*(v+z) - H^*(v) - \mbf p(v) \cdot z \right\}.
\]
Note that \eqref{DHstar} and the convexity of $H^*$ imply that
\[
	0 \le G(v) \le C(1 + |v|^{q'-1}) \quad \text{for all } v \in \RR^d.
\]

\begin{theorem}\label{T:enhancement}
	Assume that $H$ satisfies \eqref{A:Hestimates} and the random field $F$ satisfies \eqref{A:randomfield} and \eqref{A:fC1kappa}. Let $\oline{H}$ be the effective Hamiltonian from Theorem \ref{T:homog}. Then there exists $c = c(M_0,\kappa) > 0$ such that, for all $\lambda \in (0,\kappa)$,
		\begin{equation}\label{enhancementestimates}
			\oline{H}(p) \ge H(p) + \frac{c (\mbf E|Df(0)|^2)^{2 + 1/\lambda}}{1 + G(\mbf v(p))}
			\quad \text{for all } p \in \RR^d.
		\end{equation}
\end{theorem}

As an example, consider the Hamiltonian
\[
	H(p) := \frac{1}{q} |p|^q,
\]
whose Legendre transform is given by
\[
	H^*(v) = \frac{1}{q'} |v|^{q'}.
\]
We then have
\[
	\mbf p(v) = DH^*(v) = |v|^{q'-2} v, \quad \mbf v(p) = DH(p) = |p|^{q-2}p,
\]
and, for some constant $C = C(q) > 0$,
\[
	G(v) \le C(1 + |v|)^{q'-2},
\]
so that \eqref{enhancementestimates} becomes, for some $c > 0$,
\[
	\oline{H}(p) \ge \frac{1}{q}|p|^q +c (1 + |p|)^{\frac{q-2}{q-1}}.
\]

\begin{proof}[Proof of Theorem \ref{T:enhancement}]
	Fix $v \in \RR^d$, $M > 0$, and $N \in \NN$. We define a path $\gamma \in \mcl A(0,NMv,0,NM)$ as follows: set
	\[
		\eta_r := 1- \abs{2r - 1} \quad \text{for } r \in [0,1],
	\]
	and, for a sequence $\pars{ u_k }_{k \in \NN} \subset B_1 \subset \RR^d$ and
	\begin{equation}\label{deltaM}
		0 < \delta \le \frac{M}{2},
	\end{equation}
	define
	\[
		\gamma_r := vr + \delta u_k \eta\pars{ \frac{r - kM}{M} } \quad \text{for } r \in [kM,(k+1)M] \text{ and } k = 0,1,2,\ldots.
	\]
	Then
	\[
		\frac{1}{NM} L(0,NMv,0,NM,\omega) \le \frac{1}{NM}\int_0^{NM} H^*(\dot \gamma_r)dr + \frac{1}{NM}  \int_0^{NM} f(\gamma_r,\omega) \cdot dB_r(\omega).
	\]
	Since $|\dot \gamma - v| \le \frac{2\delta}{M} \le 1$ and $\frac{1}{NM}\int_0^{NM} \dot \gamma_r dr = v$, we find that
	\begin{align*}
		\frac{1}{NM}\int_0^{NM} H^*(\dot \gamma_r)dr
		&= H^*(v) + \mbf p(v) \cdot \frac{1}{NM} \int_0^{NM} \pars{ \dot \gamma_r - v}dr\\
		&+ \frac{1}{NM} \int_0^{NM} \pars{ H^*(\dot \gamma_r) - H^*(v) - \mbf p(v) \cdot(\dot \gamma_r - v)}dr\\
		&\le H^*(v) + \frac{2\delta}{M} G(v).
	\end{align*}
	For $r \ge 0$, set $B^M(r) := M^{-1/2} B(Mr)$. Then
	\begin{align*}
		\frac{1}{NM} \int_0^{NM} & f(\gamma_r,\omega) \cdot dB_r(\omega)
		- \frac{1}{NM} \int_0^{NM} f(vr,\omega) \cdot dB_r(\omega) \\
		&= \frac{\delta}{NM} \sum_{k=0}^{N-1} u_k \cdot \int_{kM}^{(k+1)M} \int_0^1 Df\pars{ vr + s\delta \eta\pars{ \frac{r - kM}{M}} u_k,\omega} ds \; \eta\pars{ \frac{r - kM}{M} } dB_r(\omega)\\
		&= \frac{\delta}{NM^{1/2}} \sum_{k=0}^{N-1} u_k \cdot \int_{0}^{1} \int_0^1 Df\pars{ Mv(r+k) + s\delta \eta\pars{r} u_k,\omega} ds \; \eta\pars{r } dB^M_{r+k}(\omega).
	\end{align*}
	The choice of the sequence $(u_k)_{k \in \NN}$ was arbitrary, and so
	\begin{equation}\label{beforelimit}
		\begin{split}
		\frac{1}{NM} L(0, NM v, 0, NM,\omega) &\le H^*(v) + \frac{2\delta}{M}G(v) \\
		&+ \frac{1}{NM} \int_0^{NM} f(vr,\omega) \cdot dB_r + \frac{\delta}{NM^{1/2}} \sum_{k=0}^{N-1} Z_\delta(v,\sigma_k \omega),
		\end{split}
	\end{equation}
	where $\sigma_k := \bs \tau_{Mvk, Mk}$ and
	\[
		Z_\delta(v,\omega) := \min_{u \in B_1} u \cdot \int_0^1 \int_0^1 Df \pars{ Mvr + s \delta \eta(r) u, \omega}ds \; \eta(r) dB^M_r(\omega).
	\]
	
	The random field $Z_\delta$ takes the form
	\[
		Z_\delta(v,\omega) = \min_{u \in B_1} u \cdot Y(v,\delta u,\omega),
	\]
	where
	\[
		Y(v,y,\omega) := \int_0^1 \int_0^1 Df\pars{ Mvr + s  y \eta(r),\omega} ds \eta(r)dB^M_r(\omega).
	\]
	In view of \eqref{A:fC1kappa}, for all $m \ge 1$, there exists a constant $C = C(m) > 0$ such that, for all $y_1,y_2 \in B_1$,
	\[
		\mbf E \abs{ Y(v,y_1,\cdot) - Y(v,y_2,\cdot)}^m \le C M_0^m  |y_1 - y_2|^{\kappa m}.
	\]
	The Kolmogorov continuity criterion implies that, for any $\lambda \in (0,\kappa)$ and $\mbf P$-almost every $\omega \in \Omega$, $Y(v,\cdot,\omega) \in C^{\lambda}(B_1)$, and, moreover, for some constant $C = C(M_0,\lambda)$,
	\[
		\sup_{v \in \RR^d} \mbf E [Y(v,\cdot,\cdot)]_{C^{\lambda}(B_1)} \le C.
	\]
	We then find that $Z_\delta(v,\cdot) \in L^1(\Omega, \mbf P)$, since
	\[
		\abs{Z_\delta(v,\omega)} \le \sup_{|u| \le 1} \abs{ Y(v,\delta u,\omega) } \le |Y(v,0,\omega)| + [Y(v,\cdot,\omega)]_{C^{\lambda}(B_1)} \delta^{\lambda}.
	\]
	Moreover,
	\[
		\mbf E Z_\delta(v,\cdot) \le \mbf E Z_0 (v,\cdot) + C \delta^{\lambda} = - \mbf E \abs{ Y(v,0,\cdot) } + C \delta^{\lambda}.
	\]
	Recall that we may assume, in view of the independence of $f$ and $B$, that $\mbf P$ takes the form $\mu \otimes \nu$ on the probability space $C^1_b(\RR^d,\RR^m) \times C([0,\oo),\RR^m)$, where $\mu$ is a probability measure that is stationary and ergodic with respect to spatial translations, and $\nu$ is the Wiener measure (see subsection \ref{SS:productmeasure}). The rotational invariance of $\nu$ then yields
	\begin{align*}
		\mbf E \abs{ Y(v,0,\cdot)} &= \int_{C^1_b(\RR^d,\RR^m)} \int_{C([0,\oo),\RR^m)} \abs{ \int_0^1 Df(Mvr)\eta(r) \cdot dB_r } d\nu(B) d\mu(f)\\
		&=  \int_{C^1_b(\RR^d,\RR^m)} \int_{C([0,\oo),\RR^m)}  \abs{ \int_0^1 \abs{Df(Mvr)}\eta(r) dB^1_r} d\nu(B) d\mu(f).
	\end{align*}
	For fixed $f \in C^1_b(\RR^d,\RR^m)$, the random variable
	\[
		C([0,\oo),\RR^m) \ni B \mapsto \int_0^1 \abs{Df(Mvr)}\eta(r) dB^1_r
	\]
	is a stochastic integral with deterministic integrand, and is therefore a Gaussian random variable with respect to the probability measure $\nu$. Its standard deviation can therefore be computed, according to It\^o's formula, as
	\[
		\int_{C([0,\oo),\RR^m)} \abs{ \int_0^1 \abs{Df(Mvr)}\eta(r) dB^1_r(\cdot) } d\nu(B) = \sqrt{ \frac{2}{\pi} }  \pars{ \int_0^1 |Df(Mvr)|^2 \eta(r)^2 dr }^{1/2}.
	\]
	Since $0 \le \eta \le 1$, we have
	\[
		\int_0^1 \abs{ Df(Mvr)}^2 \eta(r)^2dr \le \nor{Df}{\oo} \pars{ \int_0^1 \abs{ Df(Mvr)}^2\eta(r)^2dr}^{1/2}.
	\]
	It now follows from Fubini's theorem and the stationarity of $f$ that
	\begin{align*}
		\mbf E \abs{ Y(v,0,\cdot)} &= \sqrt{ \frac{2}{\pi} } \int_{C^1_b(\RR^d,\RR^m)} \pars{ \int_0^1 |Df(Mvr)|^2 \eta(r)^2 dr }^{1/2} d\mu(f) \\
		&\ge \frac{1}{\nor{Df}{\oo}} \sqrt{ \frac{2}{\pi} } \int_{C^1_b(\RR^d,\RR^m)} \int_0^1 |Df(Mvr)|^2 \eta(r)^2 dr d\mu(f) \\
		&= \frac{1}{3 \nor{Df}{\oo}} \sqrt{ \frac{2}{\pi} } \mbf E \abs{ Df(0)}^2.
	\end{align*}
	We then set
	\begin{equation}\label{deltadef}
		\delta := \oline{c} \pars{\mbf E |Df(0)|^2}^{1/\lambda}
	\end{equation}
	for some sufficiently small constant $\oline{c} = \oline{c}(M_0,\lambda) > 0$, and we conclude that, for a further constant $c = c(M_0,\kappa) > 0$,
	\[
		\sup_{v \in \RR^d} \mbf E Z_\delta(v,\cdot) \le - c \mbf E \abs{ Df(0)}^2.
	\]
	Taking the expectation of both sides of \eqref{beforelimit}, we obtain, for some $c = c(M_0,\kappa) > 0$,
	\[
		\frac{1}{NM} \mbf E L(0,NMv, 0, NM,\cdot) \le H^*(v) + \frac{2\delta}{M}G(v) + \frac{\delta}{M^{1/2} }\mbf E Z_\delta(v,\cdot) \le H^*(v) + c \delta\pars{ \frac{1}{M}G(v) - \frac{\mbf E \abs{ Df(0)}^2}{M^{1/2} }},
	\]
	and so, sending $N \to \oo$,
	\[
		\oline{L}(v) \le H^*(v) + c \delta \pars{ \frac{1}{M}G(v) - \frac{\mbf E \abs{ Df(0)}^2}{M^{1/2} }}.
	\]
	We now choose 
	\[
		M := 4 \frac{ (1 + G(v))^2}{ (\mbf E |Df(0)|^2)^{2}}.
	\]
	Observe that, if $\oline{c}$ in \eqref{deltadef} is chosen so that $\oline{c} M_0^{4 + 2/\lambda} \le 2$, then
	\[
		\delta = \oline{c}  \pars{\mbf E |Df(0)|^2}^{1/\lambda} \le \frac{\oline{c} M_0^{4+2/\lambda}}{(\mbf E|Df(0)|^2)^2} \le \frac{M}{2},
	\]
	so that \eqref{deltaM} is satisfied. For this choice of $M$, we obtain, for some constant $c = c(M_0,\kappa,\lambda) > 0$,
	\[
		\oline{L}(v) \le H^*(v) - c\delta\frac{ (\mbf E \abs{ Df(0)}^2 )^2}{1 + G(v)} = H^*(v) - c\oline{c}\frac{ (\mbf E \abs{ Df(0)}^2 )^{2 + 1/\lambda}}{1 + G(v)}.
	\]
	and therefore, for all $p \in \RR^d$,
	\[
		\oline{H}(p) \ge \sup_{v \in \RR^d} \left\{ p \cdot v - H^*(v) + c\oline{c}\frac{ (\mbf E \abs{ Df(0)}^2 )^{2+1/\lambda}}{1 + G(v)} \right\}\ge H(p) + c\oline{c}\frac{ (\mbf E \abs{ Df(0)}^2 )^{2+1/\lambda}}{1 + G(\mbf v(p))}.
	\]
\end{proof}

\section{Noise of varying strength} \label{S:differentscalings}

We conclude by studying, for $\theta \in \RR$, $\eps > 0$, and a stationary-ergodic random field $f$ and Brownian motion $B$ satisfying \eqref{A:BM} - \eqref{A:muergodic}, the initial value problem
\begin{equation}\label{E:differentscalings}
	u^\eps_t + H(Du^\eps) = \eps^\theta f\pars{ \frac{x}{\eps},\omega } \cdot \dot B(t,\omega) \quad \text{in } \RR^d \times (0,\oo) \times \Omega \quad \text{and} \quad u^\eps(x,0,\omega) = u_0(x) \quad \text{in } \RR^d \times \Omega.
\end{equation}

The strength of the noise determines the nature of the enhancement effect for vanishing $\eps$. Namely, when $\theta$ is equal to the scaling critical exponent $1/2$, the enhancement property can be exactly characterized using the results from the previous section. When $\theta > 1/2$, the noise is macroscopically insignificant, while taking $\theta < 1/2$ gives rise to infinite velocity. 

\begin{theorem}\label{T:differentscalings}
	Assume $u_0 \in BUC(\RR^d)$, \eqref{A:Hestimates}, and $f$ and $B$ satisfy \eqref{A:BM} - \eqref{A:muergodic} and \eqref{A:fC1kappa}.
	
	\begin{enumerate}[(a)]
	\item\label{T:super} If $\theta > 1/2$, then, as $\eps \to 0$, $u^\eps$ converges locally uniformly in probability to the solution $u$ of
	\[
		u_t + H(Du) = 0 \quad \text{in } \RR^d \times (0,\oo) \quad \text{and} \quad u(\cdot,0) = u_0 \quad \text{in } \RR^d.
	\]
	\item\label{T:sub} If $\theta < 1/2$ and $f$ is nonconstant, then, as $\eps \to 0$, $u^\eps$ converges locally uniformly in $\RR^d \times (0,\oo)$ in probability to $-\oo$.
	
	\item\label{T:crit} If $\theta = 1/2$, then there exists a deterministic, convex Hamiltonian $\oline{H}:\RR^d \to \RR$ satisfying \eqref{Hbarestimates} and \eqref{enhancementestimates} such that, as $\eps \to 0$, $u^\eps$ converges locally uniformly in probability to the solution $\oline{u}$ of 
	\[
		\oline{u}_t + \oline{H}(D\oline{u}) = 0 \quad \text{in } \RR^d \times (0,\oo) \quad \text{and} \quad \oline{u}(\cdot,0) = u_0 \quad \text{in } \RR^d \times \{0\}.
	\]
	\end{enumerate}
\end{theorem}

\begin{proof}
Replacing the Brownian motion $B$ with $t \mapsto \eps^{1/2} B(t/\eps)$ and invoking Lemma \ref{L:forcingstability}, it follows that it suffices to prove the appropriate local uniform limits, with probability one, for the function
\begin{equation}\label{uepsinlaw}
	\tilde u^\eps(x,t,\omega) := \inf_{y \in \RR^d} \pars{ u_0(y) + L_{\eps,f}(y,x,0,t,\omega) },
\end{equation}
where, for $\eps > 0$, $(x,y,s,t) \in \RR^d \times \RR^d \times [0,\oo) \times [0,\oo)$ with $s < t$, and $\omega \in \Omega$,
\[
	L_{\eps,f}(x,y,s,t,\omega) := \inf\left\{ \int_s^t H^*(\dot \gamma_r)dr + \eps^{\theta - 1/2} \int_s^t f(\eps^{-1} \gamma_r,\omega) \cdot dB_{r/\eps}(\omega)  : \gamma \in \mcl A(x,y,s,t) \right\}.
\]

\eqref{T:super} Similarly to the proof of Theorem \ref{T:homog}, the result will follow from two facts:
\begin{equation}\label{Lgrowth}	
	\left\{
	\begin{split}
	&\text{there exists $C > 1$ and $\eps_0 : \Omega \to \RR_+$ such that, with probability one,}\\
	&\text{for all $\eps \in (0,\eps_0)$, $x,y \in \RR^d$, and $s,t \in [0,T]$ with $s < t$,}\\
	&-C(t-s)^\alpha + \frac{1}{C} \frac{|y-x|^{q'}}{(t-s)^{q'-1}} \le L_{\eps,f}(x,y,s,t) \le C \pars{ \frac{|y-x|^{q'}}{(t-s)^{q'-1} } + (t-s)^\alpha},
	\end{split}
	\right.
\end{equation}
and
\begin{equation}\label{LtoHstar}
	\left\{
	\begin{split}
		&\text{for all $R > 0$ and $0 < \tau < T$,}\\
		&\lim_{\eps \to 0} \sup_{x,y \in B_R} \sup_{s,t \in [0,T], \; \tau \le t-s \le T}  \abs{ L_{\eps,f}(x,y,s,t) - (t-s) H^* \pars{ \frac{y-x}{t-s}} } = 0.
	\end{split}
	\right.
\end{equation}

Setting
\[
	\sigma^\eps := M_0 \eps^{1/2}, \quad \Lambda^\eps := M_0 \eps^{-1/2}, \quad F_\eps := F_{\sigma^\eps,\Lambda^\eps}, \quad \text{and} \quad \mcl D_\eps := \mcl D_{\sigma^\eps,\Lambda^\eps},
\]
we see that, by Lemma \ref{L:fauxIto}, there exists a constant $C = C(T, M_0) > 0$ such that, for all $\delta \in (0,1)$, $0 \le s \le r_1 \le r_2 \le t \le T$, $f \in F_\eps$, and Lipschitz $\gamma$,
\begin{equation}\label{thetafauxIto}
	\abs{ \int_{r_1}^{r_2} f(\eps^{-1} \gamma_r) \cdot dB_{r/\eps} } \le \pars{C \delta^{q'} \int_s^t |\dot \gamma_r|^{q'} dr + \frac{C + \mcl D_\eps}{\delta^q} }(r_2 - r_1)^\alpha,
\end{equation}
and, for some $C = C(T,M_0,p) > 0$ and all $\lambda \ge 1$,
\[
	\mbf P \pars{ \mcl D_\eps > \lambda} \le \frac{C \eps^{p/2}}{\lambda^{p/2}}.
\]

We first prove \eqref{Lgrowth}. Since $\theta > 1/2$, similar arguments as in the proof of Lemma \ref{L:minimizer} give, for some $C > 1$ and all $(x,y,s,t) \in \RR^d \times \RR^d \times [0,\oo) \times [0,\oo)$ with $s < t$ and $\eps \in (0,1)$,
\[
	\frac{1}{C} \frac{|y-x|^{q'}}{(t-s)^{q'-1}} - C\pars{ 1 + \mcl D_\eps}(t-s)^{\alpha} \le L_{\eps,f}(x,y,s,t) \le C\pars{  \frac{|y-x|^{q'} }{(t-s)^{q'-1}} + (1 + \mcl D_\eps)(t-s)^\alpha}.
\]
It follows from the Borel-Cantelli lemma that there exists $k_0: \Omega \to \NN$ such that, for a possibly different constant $C > 1$, with probability one, for all $k \ge k_0$ and all $(x,y,s,t) \in \RR^d \times \RR^d \times [0,T] \times [0,T]$ with $s < t$,
\[
	-C(t-s)^\alpha + \frac{1}{C} \frac{|y-x|^{q'}}{(t-s)^{q'-1}} \le L_{2^{-k},f}(x,y,s,t) \le C \pars{ \frac{|y-x|^{q'}}{(t-s)^{q'-1} } + (t-s)^\alpha}.
\]

Now let $0 < \eps < \eps_0 := 2^{-k_0}$, and choose $k > k_0$ so that
\[
	2^{-k-1} < \eps \le 2^{-k}.
\]
Set $\tau = 2^k \eps$, which satisfies $\tau \in (1/2,1]$. A straightforward scaling argument yields
\begin{equation}\label{monotonescaling}
	L_{\eps,f}(x,y,s,t) = \tau L_{2^{-k}, \tau^{\theta-1/2} f} \pars{ \frac{x}{\tau},\frac{y}{\tau},\frac{s}{\tau},\frac{t}{\tau}},
\end{equation}
and therefore, for yet another $C > 1$, we find that, with probability one, \eqref{Lgrowth} holds for all $(x,y,s,t)$ and $\eps \in (0,\eps_0)$.

We now establish \eqref{LtoHstar}. Fix $R > 0$ and $0 < \tau < T$. Then \eqref{thetafauxIto} implies that there exists $C = C(R,T,\tau,q,M_0) > 0$ such that, with probability one, for all $x,y \in B_R$ and $s,t \in [0,T]$ with $t - s \ge \tau$, 
\begin{align*}
	L_{\eps,f}(x,y,s,t) &\le (t-s) H^*\pars{ \frac{y-x}{t-s} } + \eps^{\theta-1/2} \int_s^t f\pars{ \frac{1}{\eps} \pars{ x + \frac{y-x}{t-s} r } }dB_{r/\eps}\\
	&\le (t-s) H^*\pars{ \frac{y-x}{t-s} } +C \eps^{\theta-1/2} \pars{ 1 + \mcl D_\eps}.
\end{align*}
For the lower bound, let $\nu \in (0,1)$ and $\gamma \in \mcl A(x,y,s,t)$ satisfy 
\[
	L_{\eps,f}(x,y,s,t) + \nu \ge \int_s^t H^*(\dot \gamma_r)dr + \eps^{\theta-1/2} \int_s^t f(\eps^{-1} \gamma_r)\cdot dB_{r/\eps}.
\]
In view of \eqref{Lgrowth}, we then have
\[
	\int_s^t |\dot \gamma_r|^{q'} dr \le C \pars{ 1 + \frac{1 + \mcl D_\eps}{\delta^q} + \delta^{q'} \int_s^t |\dot \gamma_r|^{q'}dr},
\]
and then rearranging terms and choosing $\delta$ sufficiently small yields
\[
	\int_s^t |\dot \gamma_r|^{q'}dr \le C(1 + \mcl D_\eps).
\]
It follows from Jensen's inequality that
\[
	L_{\eps,f}(x,y,s,t) + \nu \ge (t-s) H^*\pars{ \frac{y-x}{t-s} } - \eps^{\theta - 1/2} \pars{ 1 + \mcl D_\eps}.
\]
As $\nu$ was arbitrary, we conclude, combining the upper and lower bounds, that
\[
	\sup_{f \in F_{\eps}} \sup_{x,y \in B_R} \sup_{s,t \in [0,T], \; t - s \ge \tau} \abs{ L_{\eps,f}(x,y,s,t) - (t-s) H^*\pars{ \frac{y-x}{t-s} } } \le C \eps^{\theta - 1/2} (1 + \mcl D_\eps).
\]
As before, using the Borel-Cantelli lemma, the above expression converges with probability one, as $k \to \oo$ along the subsequence $\eps_k = 2^{-k}$, to $0$. The convergence over all $\eps \to 0$ can be seen by once again appealing to the scaling relationship \eqref{monotonescaling}.

(b) Let $v_0 \in \RR^d$ be such that
\[
	H^*(v_0) = \min_{v \in \RR^d} H^*(v).
\]
Then, since $\theta < 1/2$, we have, for all $\eps \in (0,1)$,
\[
	L_{\eps,f}(x,y,s,t) \le (t-s) H^*(v_0) + \eps^{\theta - 1/2} \pars{ - (t-s)H^*(v_0) + L^\eps(x,y,s,t) },
\]
where $L^\eps$ is given by \eqref{Leps}. It follows that
\begin{align*}
	\tilde u^\eps(x,t) &\le \inf_{y \in \RR^d} \pars{ u_0(y) + t H^*(v_0) + \eps^{\theta - 1/2} \pars{  - tH^*(v_0) + L^\eps(y,x,0,t) } }\\
	&\le u_0(x - tv_0) + t H^*(v_0) + \eps^{\theta - 1/2} \pars{ - t H^*(v_0) + L^\eps(x - tv_0,x,0,t)}.
\end{align*}
Lemma \ref{L:Lbarlocaluniform} yields that, with probability one, locally uniformly in $\RR^d \times (0,\oo)$,
\[
	L^\eps(x - tv_0,x,0,t) \xrightarrow{\eps \to 0} t \oline{L}(v_0) < t H^*(v_0),
\]
where the strict inequality is due to Theorem \ref{T:enhancement}. The result follows.

(c) This is a consequence of Theorems \ref{T:homog} and \ref{T:enhancement}.

\end{proof}

\section*{Acknowledgements}

The author would like to thank Panagiotis Souganidis and Pierre Cardaliaguet for many helpful comments and suggestions in the preparation of this work.

\bibliography{homogstochforcedHJ}{}
\bibliographystyle{acm}

\end{document}